\documentclass[12pt]{article}
\usepackage{xcolor,caption,subcaption,graphicx}
\usepackage{amsmath}
\usepackage{tikz}
\usepackage{amsthm}
\usepackage{xchicago}
\makeatother
\newcommand {\citeAY} [1] {\citeNP {#1}}%
\newcommand {\citeAPY}[1] {\citeN  {#1}}%
\renewcommand {\showoriginalref}[1]{}
\renewcommand {\showCODEN}[1]{}
\renewcommand {\showISSN}[1]{}
\renewcommand {\showMR}[3]{}

\newcommand\eq[1] {(\ref{#1})}

\newcommand\sect[1] {\ref{sec:#1}}

\newcommand\labsect[1] {\label{sec:#1}}

\newcommand{\bfm}[1]{\mbox{\boldmath ${#1}$}}

\newcommand{\beqa}{\begin{eqnarray}}
\newcommand{\eeqa}[1]{\label{#1}\end{eqnarray}}
\newcommand{\beq}{\begin{equation}}
\newcommand{\eeq}[1]{\label{#1}\end{equation}}
\newcommand{\bbm}{\begin{bmatrix}}
\newcommand{\ebm}{\end{bmatrix}}
\newcommand{\bpm}{\begin{pmatrix}}
\newcommand{\epm}{\end{pmatrix}}

\newcommand{\Divop}{\mathrm{div \,}}
\newcommand{\Curlop}{\mathrm{curl\,}}

\newcommand{\Real}{\mathop{\rm Re}\nolimits}
\newcommand{\Imag}{\mathop{\rm Im}\nolimits}

\newtheorem{Thm}{Theorem}

\newtheorem{Def}[Thm]{Definition}
\newtheorem{Pro}[Thm]{Proposition}
\newtheorem{Cor}[Thm]{Corollary}
\newtheorem{Rem}[Thm]{Remark}

\newcommand{\dint}{\mathrm{d}}

\def\Bpsi{{\psi}}
\def\Bphi{{ \phi}}

\newcommand{\Gve}{\varepsilon}

\newcommand{\Gm}{\mu}

\newcommand{\Gs}{\sigma}

\newcommand{\Go}{\omega}

\newcommand{\BGve}{\bfm\varepsilon}
\newcommand{\BGf}{\bfm\phi}

\newcommand{\BGk}{\bfm\kappa}

\newcommand{\BGm}{\bfm\mu}

\newcommand{\BGr}{\bfm\rho}

\newcommand{\BGy}{\bfm\psi}

\newcommand{\BGG}{\bfm\Gamma}

\newcommand{\BGY}{\bfm\Psi}

\newcommand{\CU}{{\cal U}}

\newcommand{\BCI}{{\bfm{\cal I}}}

\newcommand{\BCT}{{\bfm{\cal T}}}

\def\dd{{\rm d}}

\def\Be{{\bf e}}
\def\Bf{{\bf f}}
\def\Bg{{\bf g}}

\def\Bm{{\bf m}}
\def\Bn{{\bf n}}

\def\Bu{{\bf u}}
\def\Bv{{\bf v}}
\def\Bw{{\bf w}}
\def\Bx{{\bf x}}
\def\By{{\bf y}}
\def\Bz{{\bf z}}
\def\BA{{\bf A}}
\def\BB{{\bf B}}
\def\BC{{\bf C}}

\def\BE{{\bf E}}
\def\BF{{\bf F}}

\def\BH{{\bf H}}
\def\BI{{\bf I}}
\def\BJ{{\bf J}}
\def\BK{{\bf K}}
\def\BL{{\bf L}}
\def\BM{{\bf M}}
\def\BN{{\bf N}}

\def\BP{{\bf P}}
\def\BQ{{\bf Q}}

\def\BT{{\bf T}}

\def\BV{{\bf V}}

\def\BY{{\bf Y}}
\def\BZ{{\bf Z}}

\def \ep {\varepsilon}

\def \ba {\begin{array}}
\def \ea {\end{array}}

\def \refe #1.{(\ref{#1})}
\def \reff #1.{figure~\ref{#1}}
\def \refs #1.{section~\ref{#1}}
\def \refss #1.{subsection~\ref{#1}}
\def \refD #1.{Definition~\ref{#1}}
\def \refT #1.{Theorem~\ref{#1}}
\def \refL #1.{Lemma~\ref{#1}}
\def \refC #1.{Corollary~\ref{#1}}
\def \refP #1.{Proposition~\ref{#1}}
\def \refR #1.{Remark~\ref{#1}}
\def \refE #1.{Example~\ref{#1}}
\def \refN #1.{Notation~\ref{#1}}
\newcommand{\bbR}{{\mathbb{ R}}}
\newcommand{\bbC}{{\mathbb{C}}}

\usepackage{graphics}
\usepackage{amssymb}
\begin{document}

\vspace{-1in}
\title{Analyticity of the Dirichlet-to-Neumann map for the time-harmonic Maxwell's equations}%
\label{chap:analyticity}
\author{Maxence Cassier$^{*}$, Aaron Welters$^{**}$ and Graeme W. Milton$^{*}$}
\date{$^{*}$\small{Department of Mathematics, University of Utah, Salt Lake City, UT 84112, USA} \\
$^{**}$\small{Department of Mathematical Sciences, Florida Institute of Technology, Melbourne, FL 32901, USA}\\
\small{emails: cassier@math.utah.edu, awelters@fit.edu, milton@math.utah.edu}}
\vspace{2ex}
\maketitle
\begin{abstract}
In this chapter we derive the analyticity properties of the electromagnetic Dirichlet-to-Neu\-mann map%
\index{Dirichlet-to-Neumann map!analytic properties}
for the time-harmonic Maxwell's equations%
\index{Maxwell's equations}
for passive linear multicomponent media. Moreover, we discuss the connection of this map to Herglotz functions%
\index{Herglotz functions}
for isotropic and anisotropic multicomponent composites.
\end{abstract}
\vspace{3ex}
\noindent
\begin{mbox}
\end{mbox}
\vspace{3ex}
\vskip2mm

\section{Introduction}
\setcounter{equation}{0}
In this chapter, we study the analytic properties of the electromagnetic ``Dirichlet-to-Neumann'' (DtN) map for a composite material. Using passive linear multicomponent media, we will
prove that this DtN map is an analytic function of the  dielectric permittivities and magnetic permeabilities (multiplied by the frequency $\omega$) which characterize each phase. More specifically, it belongs to a special class of
functions known as Herglotz functions. In that sense, this chapter is highly connected to the previous one by Graeme Milton since both are proving analyticity properties on the DtN map, but with different methods. In \cite[Chapter 3]{Graeme:2016:ETC}, these analyticity properties are derived
by using the theory of composite materials, whereas in this chapter they are proved via spectral theory in the usual functional framework associated with the time-harmonic Maxwell's equations. Maxwell's equations at fixed frequency $\Go$ involve the electric permittivity $\Gve(\Bx,\Go)$ (also called the dielectric constant if measured relative to the
permittivity of the vacuum) and the magnetic permeability $\Gm(\Bx,\Go)$.
The approach taken in the current chapter has the important advantage of being applicable to bodies where the moduli $\omega\Gve(\Bx,\Go)$
and $\omega\Gm(\Bx,\Go)$ are not piecewise constant but instead vary smoothly (or not) with position. In this case we establish (in Subsection \sect{subsec.ext}) the
Herglotz properties of the Dirichlet-to-Neumann map, as a function of frequency, assuming the material is passive at each point $\Bx$, i.e., that $\omega\Gve(\Bx,\Go)$
and $\omega\Gm(\Bx,\Go)$ are Herglotz functions of the frequency $\omega$.

The use of theory of Herglotz functions in electromagnetism and in the theory of composites has  many important impacts and
 consequences (Bergman \citeyearNP{Bergman:1978:DCC}, \citeyearNP{Bergman:1980:ESM}, \citeyearNP{Bergman:1982:RBC}; Milton \citeyearNP{Milton:1980:BCD}, \citeyearNP{Milton:1981:BCP}, \citeyearNP{Milton:1981:BTO}, \citeyearNP{Milton:2002:TOC}; \citeAY{Golden:1983:BEP}; \citeAY{DellAntonio:1986:ATO}; \citeAY{Bruno:1991:ECS};  Lipton \citeyearNP{Lipton:2000:OBE}, \citeyearNP{Lipton:2001:OIG}; \citeAY{Gustafsson:2010:SRP}; \citeAY{Bernland:2011:SRC}; \citeAY{Liu:2013:CPP}; \citeAY{Welters:2014:SLL}) especially in developing bounds on certain physical quantities. Based on this and the work of
\citeAPY{Golden:1985:BEP}, \citeAPY{Bergman:1986:EDC}, Milton (\citeyearNP{Milton:1987:MCEa}, \citeyearNP{Milton:1987:MCEb}) and \citeAPY{Milton:1990:RCF} on developing bounds on effective tensors of composites containing more than two phases
using analyticity of the effective tensors as a multivariable function of the moduli of the phases, we also establish that the DtN map is an analytic function of  the permittivity and permeability tensors of each phase. Another potential application of these analytic properties is to derive information about the DtN map for real frequencies by using the theory of boundary-values of Herglotz functions
(for instance, see \citeAY{Gesztesy:2000:MVH} and \citeAY{Naboko:1996:ZTB}).
Moreover, as the DtN map is usually used as data in electromagnetic inverse problems%
\index{inverse problems}
(see, for instance, \citeAY{Albanese:2006:ISP}; \citeAY{Uhlmann:2014:REO}, \citeAY{Ola:1993:IBV}), we believe these analyticity properties and the connection to the theory of Herglotz functions will have important applications in this area of research (see \cite[Chapter 5]{Graeme:2016:ETC}). The Herglotz properties might also be important to characterize the complete set of all possible Dirichlet-to-Neumann maps (either at fixed frequency or as a function of frequency) associated with multiphase bodies with frequency independent permittivity and permeability. Indeed such analyticity
properties were a key ingredient to characterize the possible dynamic response functions of multiterminal mass-spring networks (\citeAY{Vasquez:2011:CCS}). These response functions
are the discrete analogs of the Dirichlet-to-Neumann map in that problem. Additionally, analytic properties were a key ingredient in the theory of exact relations%
\index{composites!exact relations}
(\citeAY{Grabovsky:1998:EREa}; \citeAY{Grabovsky:1998:EREb}; \citeAY{Grabovsky:2000:ERE}: see also Chapter 17 in \citeAY{Milton:2002:TOC} and \citeAY{Grabovsky:2004:AGC})
satisfied by the effective tensors of composites, and for establishing links between effective tensors. These are generally nonlinear relations that are microstructure independent
and thus, besides their intrinsic interest, are useful as benchmarks for numerical methods and approximations. They become linear (\citeAY{Grabovsky:1998:EREa})
after a suitable fractional linear matrix transformation is made
(which is nonunique and involves an arbitrary unit vector $\Bn$). After any such transformation is made and once certain algebraic relations are satisfied (for all unit vectors $\Bn$)
it can be proved  that all terms in the series expansion satisfy the exact relation, and then analyticity is needed to prove the relation holds
(in the domain of analyticity) even if the series expansion does not converge (\citeAY{Grabovsky:2000:ERE}).

We split this chapter in three sections. In the first one, we consider the electromagnetic DtN map for a layered media. In this setting, the DtN map can be expressed explicitly in terms of the transfer matrix%
\index{transfer matrices}
associated with the medium. This gives a good example in which one can see these analytic properties in the context of matrix perturbation theory%
\index{matrix perturbation theory}
(\citeAY{Baumg:1985:APT}; \citeAY{Kato:1995:PTO};
\citeAY{Welters:2011:ERF}). In the second section, we restrict ourselves to bounded media but with a large class of different geometries, more precisely, Lipschitz domains. In this case, using a variational reformulation of the time-harmonic Maxwell's equations (\citeAY{Cessenat:1996:MME}; \citeAY{Kirsch:2015:MTT}; \citeAY{Monk:2003:FEM}; \citeAY{Nedelec:2001:AEE}),
we prove both the well-posedness and the analyticity of the DtN map. Also we consider bodies where the moduli $\omega\Gve(\Bx,\Go)$
and $\omega\Gm(\Bx,\Go)$ are not piecewise constant but instead vary with position, and at each point $\Bx$ are Herglotz functions of the frequency $\omega$. In this case we establish the
Herglotz properties of the Dirichlet-to-Neumann map, as a function of frequency.
In both sections, the key step to prove the multivariable analyticity is Hartogs' Theorem%
\index{Hartogs' theorem}
from complex analysis which essentially says that analyticity in each variable separately implies joint analyticity (see Theorem \ref{FThm.Hartog} below).
Concerning the connection to Herglotz functions, an energy balance equation is derived (which is essentially Poynting's Theorem%
\index{Poynting's theorem}
for complex frequencies) that allows us to prove that the imaginary part of the DtN map is positive definite,  as a consequence of the positivity of the imaginary part of the material tensors. Nevertheless, in the case of anisotropic media, the connection to Herglotz functions has to be made more precise. Indeed, we leave here the usual framework of Herglotz functions of
scalar variables since we are concerned with dielectric permittivity and magnetic permeability tensors as input variables. Thus, the purpose of the last section is to provide a rigorous definition of
Herglotz functions in this general framework, that provides an alternative to the one developed in Section 18.8 of \citeAPY{Milton:2002:TOC}, by connecting this notion to the theory of holomorphic functions on tubular domains%
\index{tubular domain}
with nonnegative imaginary part as described in \citeAY{Vladimirov:2002:MTG} (see Sections $17$--$19$). This new link is especially significant since this class of functions
(like the Herglotz functions introduced in Section 18.8 of \citeAPY{Milton:2002:TOC})
admits integral representations analogous to Herglotz functions of one complex variable (the representation in the one variable case as described in Theorem \ref{thm.repHerg} below) and are deeply connected to the theory of multivariate passive linear systems%
\index{multivariate passive linear systems}
(see Section $20$ in \citeAY{Vladimirov:2002:MTG}) with the notions of convolutions, passivity, causality, Laplace/Fourier transforms, and analyticity properties.

This chapter is essentially self-contained, and written in a rigorous mathematical style. Care has been taken to explain most technical definitions so that it should
be accessible to non-mathematicians. 

Before we proceed, let us introduce some notation, definitions and theorems used in this chapter.
We denote:
\begin{itemize}
\item the complex upper-half plane by $\bbC^{+}=\left\{ z\in \bbC \mid \operatorname{Im}z>0  \right\}$,
\item the Banach space%
\index{Banach space}
of all $m\times n$ matrices with complex entries by $M_{m,n}(\bbC)$ equipped with any norm, with the square matrices $M_{n,n}(\bbC)$ denoted by $M_{n}(\bbC)$, and we treat $\bbC^n$ as $M_{n,1}(\bbC)$ (recall that a Banach space is a complete normed vector space: unlike a Hilbert space, it does not necessarily have an inner product defined on the space, just a norm.)
\item by  $\cdot$ the operation defined for all vectors $\Bu,\Bv \in \bbC^n$ via $ \Bu \cdot \,\Bv=\Bu^{T} \Bv=u_iv_i$, where $T$ denotes the transpose. Note that there is no complex conjugation in this definition, so $\Bu\cdot\Bu$ is not generally real.
\item the open, connected, and convex subset of $M_{n}(\bbC)$ of matrices with positive definite imaginary part by $$M_{n}^{+}(\bbC)=\left\{\BM\in M_{n}(\bbC) \mid \operatorname{Im} \BM>0\right\},$$ where $\operatorname{Im}\BM=(\BM-\BM^{*})/(2i)$ with $\BM^*=\overline{\BM}^{T}$ the adjoint of $\BM$, and the inequality $\BM>0$ holds in the sense of quadratic forms.
We remark that this set is invariant by the operation: $\BM \to -\BM^{-1}$  since if $\BM \in M_{n}^{+}(\bbC)$ then $\BM$ is invertible and  $$-\operatorname{Im}(\BM^{-1})= (\BM^{-1})^{*}\,\operatorname{Im}(\BM) \, \BM^{-1}>0 $$
\item by $L(E,F)$ the Banach space of all continuous linear operators from a Banach space $E$ to a Banach space $F$ equipped with the operator norm.
\end{itemize}

\begin{Def}(Analyticity)
Let $E$ and $F$ be two complex Banach spaces and  $U$ be an open set of $E$.  A function $f:U \to F$ is said to be a analytic if it is differentiable on $\CU$.
\end{Def}

\begin{Def}\label{FDef.Herg1}(Herglotz functions)
Let $m,n,N\in \mathbb{N}$, where $\mathbb{N}$ is the set of natural numbers (positive integers), and $\mathcal{T}= (\bbC^{+})^{n}$ or $(M_{N}^{+}(\bbC))^{n}$. An analytic function $h:\mathcal{T}\to \bbC$ or $h:\mathcal{T}\to  M_m(\bbC)$ is called a Herglotz function%
\index{Herglotz functions}
(also called Pick or Nevanlinna function)%
\index{Pick functions|see{Hergotz functions}}%
\index{Nevanlinna functions|see{Herglotz functions}}
if
$$
\operatorname{Im}(h(\Bz)) \geq  0, \ \forall \Bz \in \mathcal{T}.
$$
\end{Def}
We note here that Definition \ref{FDef.Herg1} is the standard definition of a Herglotz function when $\mathcal{T}=\bbC^{+}$ (see \citeAY{Gesztesy:2000:MVH}, \citeAY{Berg:2008:SPBS}) and $\mathcal{T}=(\bbC^{+})^{n}$ (in \citeAY{Agler:2012:OMF} it is called a Pick function), but not when $\mathcal{T}=(M_{N}^{+}(\bbC))^{n}$.
Its justification in this last case is given in Section \ref{sec.anisopHerg}.

A particular and useful property of Herglotz functions defined on a scalar variable, which has been a key-tool to use analytic methods to derive bounds, is the following representation theorem.%
\index{Herglotz functions!integral representation}
\begin{Thm}\label{thm.repHerg}
A necessary and sufficient condition for a function $h:\bbC^{+}\to \bbC$ to be a Herglotz function is that there exist $\alpha,\beta\in\bbR$ with $\alpha\geq 0$ and a positive regular Borel measure $\mu$ for which   $\int_{\bbR} \
 \dd \mu(\lambda)/(1+\lambda^2)$ is finite such that
\begin{equation}\label{eq.defhergl}
h(z)=\alpha \, z+\beta + \displaystyle \int_{\bbR} \left(  \frac{1}{\lambda-z}- \frac{\lambda}{1+\lambda^2}\right)\dd \mu( \lambda), \ \mbox{ for } z \in \bbC^{+}.
\end{equation}
\end{Thm}
For an extension of this representation theorem, for instance, in the case of matrix-valued Herglotz functions $h:\bbC^{+}\to M_{m}(\bbC)$, we refer to \citeAPY{Gesztesy:2000:MVH}.

\begin{Thm}(Hartogs' Theorem)\label{FThm.Hartog}%
\index{Hartogs' theorem}
If $h:U\to E$ is a function with $U$ an open subset of $\bbC^{n}$ and $E$ is a Banach space then $h$ is a multivariate analytic function (i.e., jointly analytic) if and only if it is an analytic function of each variable separately.
\end{Thm}
A proof of Hartogs' Theorem when $E=\bbC$ can be found in \citeAPY{Hormander:1990:ICA} (see Section \ 2.2, p.\ 28, Theorem 2.2.8). For the general case, we refer the reader to \citeAPY{Mujica:1986:CAB} (see Section 36, p.\ 265, Theorem 36.1).
\begin{Thm}\label{FTh.analyinvop}
Let $E$ and $F$ denote two Banach spaces and $U$ an open subset of $\bbC^{n}$. If $h:U\to L(E,F)$ is an analytic function and for each $\Bz\in U$ the value $h(\Bz)$ is an isomorphism, then the function $\Bz\to h(\Bz)^{-1}$ is analytic from $U$
into $L(F,E)$.
\end{Thm}
For a proof of Theorem \ref{FTh.analyinvop} when $n=1$, we refer the reader to \citeAPY{Kato:1995:PTO} (see Chapter 7, Section 1, pp.\ 365--366).
The proof for an integer $n>1$ is then obtained by using Hartogs' Theorem.

The next theorem, which is a rewriting of Theorem 3.12 of \citeAPY{Kato:1995:PTO} shows that the notion of weak analyticity of a family of operators in $L(E,F)$ implies the analyticity of this family for the operator norm of $L(E,F)$. More precisely, we have the following result:
\begin{Thm}\label{Th.opanalytic}
Let $E$ and $F$ be two Banach spaces, $U$ an open subset of $\bbC$ and $h:U\to L(E,F)$.
We denote by $\left\langle \cdot , \cdot \right\rangle$ the duality product of $F$ and its dual $F^{*}$.
If the function $$h_{\phi,\psi}(z)=\left\langle h(z) \phi, \psi \right\rangle, \, \forall z\in U,$$  is analytic on $U$ for all $\phi$ in a
dense subset of $E$ and for all $\psi$ in a dense subset of $F^{*}$, then $h$ is analytic in $U$ for the operator norm of  $L(E,F)$.
\end{Thm}
The following is a theorem for taking the derivative under the integral of a function which depends analytically on a complex parameter (see \citeAY{Mattner:2001:CDI}). It introduces the notion of a measure space that we briefly recall here. A measure space $(\Omega,\mathcal{F},\mu)$ is roughly speaking a triple composed of a set $\Omega$, a collection $\mathcal{F}$ of subsets of $\Omega$ that one wants to measure ($\mathcal{F}$ is called a $\Gs-$algebra) and a measure $\mu$ defined on $\mathcal{F}$.
\begin{Thm}\label{Th.derivsum}
Let $(\Omega,\mathcal{F},\mu)$ be a measure space, let $U$ be an open set of $\bbC$ and  $f:\Omega \times U \to \bbC$ be a function subject to the following assumptions:
\begin{itemize}
\item $f(\cdot,z)$ is $\mathcal{F}$ measurable for all $z\in U$ and $f(\Bx, \cdot )$ is analytic for almost every $\Bx$ in $\Omega$,
\item $\int_{\Omega}| f(\Bx,\cdot)| \, \dint \mu (\Bx)$ is locally bounded, that is, for every $z_0 \in U$ there exists a $\delta>0$ such that
$$
\sup_{z\in U \mid |z-z_0|\leq\delta } \int_{\Omega}| f(\Bx,z)| \, \dint \mu (\Bx) < \infty,
$$
\end{itemize}
then the function $F:U\to \bbC$ defined by
$$
F(z)=\int_{\Omega} f(\Bx,z) \, \dint  \mu(\Bx),
$$
is analytic in $U$ and  one can take derivatives under the integral sign:
$$
F^{(k)}(z)=\int_{\Omega} \frac{\partial^{k} f(\Bx,z)}{\partial z^{k}} \, \dint \mu(\Bx),\  \forall k \in \mathbb{N}.
$$
\end{Thm}

\section{Analyticity of the DtN map for layered media}
\setcounter{equation}{0}
\subsection{Formulation of the problem}\label{SecFormProblemLayered}

We consider passive linear two-component layered media (material 1 with moduli $\BGve_1, \BGm_1$; material 2 with moduli $\BGve_2,\BGm_2$) with layers normal to the $z$-axis. The geometry of this problem, as illustrated in \ref{Ffig:layered}, is as follows:
First, a layered medium%
\index{layered medium}
in the region $\Omega=\Omega_1\cup\Omega_2=[-d,d]$ consisting of a two-phase material lies between $z=-d$ and $z=d$. A homogeneous passive linear material lies between $-d_2\leq z \leq d_2$ (denote this ``inner'' region by $\Omega_2=[-d_2,d_2]$) with permittivity and permeability $\BGve_2,\BGm_2$. Another homogeneous passive linear material lies between $-d\leq z<-d_2$, i.e., the region $\Omega_{1,-}=[-d,-d_2)$, and $d_2<z\leq d$, i.e., the region $\Omega_{1,+}=(d_2,d]$ (denote ``outer'' region by $\Omega_1=\Omega_{1,-}\cup\Omega_{1,+}$) with permittivity and permeability  $\BGve_1,\BGm_1$. The unit outward pointing normal vectors to the boundary surfaces of these regions are $\textbf{n}\in \{\textbf{e}_3,-\textbf{e}_3\}$, where $\textbf{e}_3=\begin{bmatrix}0 & 0 & 1\end{bmatrix}^{T}$.

The dielectric permittivity $\BGve$ and magnetic permeability $\BGm$ are $3\times 3$ matrices that depend on the frequency $\omega$ and the spatial variable $z$ only (i.e., spatially homogeneous in each layer) which are defined by
\begin{align}
\BGve=\BGve(\omega,z)=\chi_{\Omega_1}(z)\BGve_1(\omega)+\chi_{\Omega_2}(z)\BGve_2(\omega),\;\;z\in[-d,d],\;\omega\in\bbC^{+},\label{DefDielecMagnMatrices1}\\
\BGm=\BGm(\omega,z)=\chi_{\Omega_1}(z)\BGm_1(\omega)+\chi_{\Omega_2}(z)\BGm_2(\omega),\;\;z\in[-d,d],\;\omega\in\bbC^{+}.\label{DefDielecMagnMatrices2}
\end{align}
Here $\chi_{\Omega_j}$ denotes the indicator function of the region $\Omega_j$, for $j=1,2$. Moreover, they have the passivity properties%
\index{passivity}
[see, for example, \cite[Section 1.6]{Graeme:2016:ETC}]
\begin{gather}
\operatorname{Im}(\omega\BGve(\omega,z))>0,\;\;\operatorname{Im}(\omega\BGm(\omega,z))>0,\;\textnormal{for } \operatorname{Im}\omega>0, \label{FPassiveLayeredMedia}
\end{gather}
and $\BGve, \BGm$ are analytic functions of $\omega$ in the complex upper-half plane for each fixed $z$, i.e.,
\begin{align}\label{FDielecMagnHerglotz}
\omega\BGve_j(\omega),\omega\BGm_j(\omega):\bbC^{+}\rightarrow M_3^{+}(\bbC) \textnormal{ are Herglotz functions, for } j=1,2.%
\index{Herglotz functions}
\end{align}

The time-harmonic Maxwell's equations%
\index{Maxwell's equations}
in Gaussian units without sources are
\begin{align}\label{FMaxwellEqsTimeHar}
\Curlop\textbf{E}=\frac{i\omega}{c}\textbf{B},\;\;\Curlop\textbf{H}=-\frac{i\omega}{c}\textbf{D},\;\;\textbf{D} = \BGve\textbf{E},\;\;\textbf{B} = \BGm\textbf{H},
\end{align}
where $c$ denotes the speed of light in a vacuum.

\begin{figure}[t]
\centering
\includegraphics[scale=1.6]{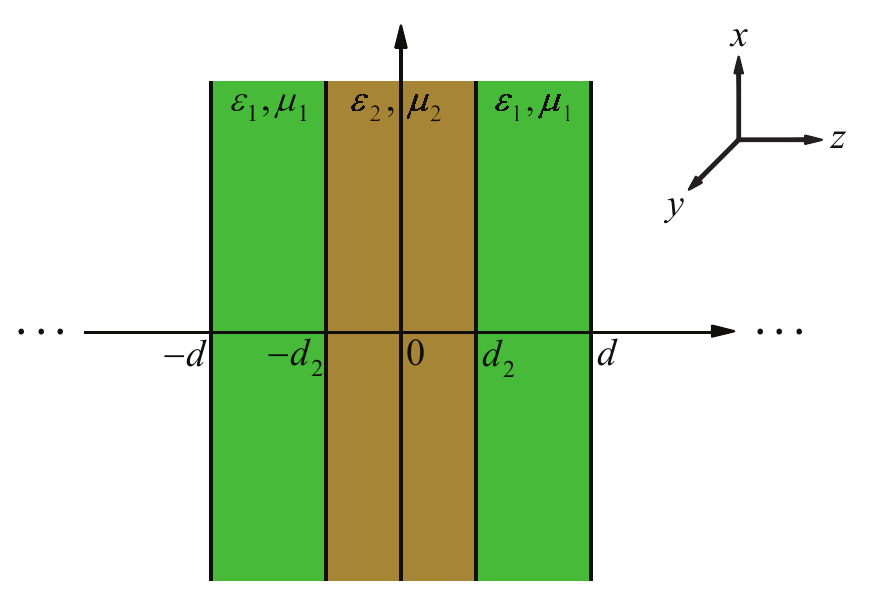}
\caption[The two-component layered medium in a sandwich configuration.]{A plane-parallel, two-component layered medium $\Omega$ consisting of two phases, $\BGve_1, \BGm_1$ and $\BGve_2,\BGm_2$, of linear passive materials with layers normal to the $z$-axis. The core containing the homogeneous material $2$ (with permittivity $\BGve_2$ and permeability $\BGm_2$) is sandwiched between the shell containing the homogeneous material $1$ (with permittivity $\BGve_1$ and permeability $\BGm_1$). Moreover, the system is symmetric about the $xy$-plane.}
\label{Ffig:layered}
\end{figure}

Let us now introduce some classical functional spaces associated to the study of Maxwell's equations (\ref{FMaxwellEqsTimeHar}) in layered media:
\begin{itemize}
\item For a bounded interval $I\subseteq \bbR$, we denote by $L^1(I)$,%
\index{Banach space!L@$L^1(I)$}
the Lebesgue space of integrable functions on $I$. It is a Banach space with norm
\begin{align}
||f||_1=\int_I|f(z)|dz,\;f\in L^1(I).
\end{align}
\item For a bounded interval $I\subseteq \bbR$, we denote by $AC(I)$,%
\index{Banach space!AC@$AC(I)$}
the Banach space of all absolutely continuous functions equipped with the norm
\begin{align}
||f||_{1,1}=\int_I|f(z)|dz+\int_I|f^{\prime}(z)| dz,\;f\in AC(I).
\end{align}
Recall, that any $f\in AC(I)$ is continuous on $I$ into $\bbC$, differentiable almost everywhere on $I$ (i.e., except on a set of Lebesgue measure zero), and is given in terms of its derivative $f^{\prime}=\frac{df}{dz}$ (which is integrable on $I$) by
\begin{align}
f(z)=f(z_0)+\int_{z_0}^{z}f^{\prime}(u)du,\;z_0,z\in I.
\end{align}
\item Denote the Banach space of all $m\times n$ matrices with entries in the Banach space $E$ with norm $||\cdot||$, where $(E,||\cdot||)\in \{(L^1(I),||\cdot||_1),(AC(I),||\cdot||_{1,1}),(\bbC,|\cdot|)\}$, by $M_{m,n}(E)$ and equipped with norm
\begin{align}
||\BM||=\left(\sum_{i=1}^m\sum_{j=1}^n||M_{ij}||^2\right)^{\frac{1}{2}},\;\BM=[M_{ij}]\in M_{m,n}(E),
\end{align}
with $M_{n,n}(E)$ denoted by $M_{n}(E)$, and we treat $E^n$ as $M_{n,1}(E)$.
\item Similar to $AC(I)$, any $\BM\in M_{m,n}(AC(I))$ is continuous on $I$, differentiable almost everywhere on $I$, and in terms of its derivative $\BM^{\prime}=\frac{d\BM}{dz}=[M_{ij}^{\prime}]$ is given by
\begin{align}
\BM(z)=\BM(z_0)+\int_{z_0}^{z}\BM^{\prime}(u)du=\left[M_{ij}(z_0)+\int_{z_0}^{z}M_{ij}^{\prime}(u)du\right],\;z_0,z\in I.
\end{align}
\item Denote the standard inner product%
\index{inner product}
on $\bbC^n$ by $(\cdot,\cdot):\bbC^n\times\bbC^n\rightarrow\bbC$, where
\begin{align}
(\psi_1,\psi_2)=\psi_1^{T}\overline{\psi_2},\;\;\psi_1,\psi_2\in\bbC^n.
\end{align}
\end{itemize}

Now, because of the translation invariance of the layered media in the $x,y$ coordinates, solutions of equation (\ref{FMaxwellEqsTimeHar}) are sought in the form

\begin{equation}\label{FTangentialBlochWaves}
\begin{bmatrix}
\textbf{E}\\
\textbf{H}
\end{bmatrix}
=
\begin{bmatrix}
\textbf{E}(z)\\
\textbf{H}(z)
\end{bmatrix}
e^{i(k_1x+k_2y)},\;x,y\in \bbR,\;z\in[-d,d],\;\BGk=(k_1,k_2)\in\bbC^2,\;\omega\in\bbC^+,
\end{equation}
in which $\BGk$ is the tangential wavevector. Maxwell's equations%
\index{Maxwell's equations!layered media}
(\ref{FMaxwellEqsTimeHar}) for this type of solution can be reduced [see the appendix
in \citeAPY{Shipman:2013:RES} and also \citeAPY{Berreman:1972:OSA} for more details] to an ordinary linear differential equation (ODE) for the vector of
tangential electric and magnetic field components $\BGy$, where
\begin{align}
\BGy(z)&=\begin{bmatrix} E_1(z) & E_2(z) & H_1(z) & H_2(z) \end{bmatrix}^{T}, \label{FTangentialComponents}\\
-i\BJ\frac{d\BGy}{dz}&=\BA(z)\BGy(z),\;\;\BGy\in (AC([-d,d]))^4,\label{FMaxwellODEs}
\end{align}
in which
\begin{align}
\BJ&=
\begin{bmatrix}
0 & \BGr\\
\BGr^* & 0
\end{bmatrix},\;\;
\BGr=\begin{bmatrix}
0 & 1\\
-1 & 0
\end{bmatrix},\;\;
\BJ^*=\BJ^{-1}=\BJ,\label{FJRhoMatrix}\\
\BA&=\BA(z)=\BA(z, \BGk, \omega \BGve_1(\omega),\omega \BGve_2(\omega), \omega \BGm_1(\omega),\omega \BGm_2(\omega)),\;z\in[-d,d],\;\BGk\in\bbC^2,\;\omega\in\bbC^+,\label{FAMatrix}.
\end{align}
Here $\BA=\BA(z)$ is a piecewise constant function of $z$ into $M_4(\bbC)$ (for fixed $\BGk$, $\omega$) with the following explicit representation in terms of the entries of the matrices $\BGve=[\ep_{ij}]$, $\BGm=[\mu_{ij}]$ in (\ref{DefDielecMagnMatrices1}), (\ref{DefDielecMagnMatrices2}):
\begin{eqnarray}\label{FAMatrixExplicit1}
\BA=\BV_{\parallel\parallel}-\BV_{\parallel\bot}\left(  \BV_{\bot\bot}\right)
^{-1}\BV_{\bot\parallel},
\end{eqnarray}
where
\begin{eqnarray}
&\BV_{\bot\bot}=
\frac{1}{c}\begin{bmatrix}
\omega\ep_{33} & 0\\
0 & \omega\mu_{33}
\end{bmatrix}
,\label{FAMatrixExplicit2}\\
&\BV_{\parallel\parallel}=
\frac{1}{c}\begin{bmatrix}
\omega\ep_{11} & \omega\ep_{12} & 0 & 0\\
\omega\ep_{21} & \omega\ep_{22} & 0 & 0\\
0 & 0 & \omega\mu_{11} & \omega\mu_{12}\\
0 & 0 & \omega\mu_{21} & \omega\mu_{22}
\end{bmatrix}
,\label{FAMatrixExplicit3}\\
&\BV_{\parallel\bot}=\frac{1}{c}\left[
\begin{array}
[c]{cc}%
\omega\ep_{13} & 0\\
\omega\ep_{23} & 0\\
0 & \omega\mu_{13}\\
0 & \omega\mu_{23}
\end{array}
\right]  +\left[
\begin{array}
[c]{cc}
0 & k_{2}\\
0 & -k_{1}\\
-k_{2} & 0\\
k_{1} & 0
\end{array}
\right]  ,\label{FAMatrixExplicit4}\\
&\BV_{\bot\parallel}=\frac{1}{c}\left[
\begin{array}
[c]{cccc}
\omega\ep_{31} & \omega\ep_{32} & 0 & 0\\
0 & 0 & \omega\mu_{31} & \omega\mu_{32}
\end{array}
\right]  +\left[
\begin{array}
[c]{cccc}
0 & 0 & -k_{2} & k_{1}\\
k_{2} & -k_{1} & 0 & 0
\end{array}
\right].\label{FAMatrixExplicit5}
\end{eqnarray}
From these matrices the normal electric and magnetic field components $\BGf$ are given in terms of their tangential components by
\begin{align}
\BGf&=\begin{bmatrix} E_3 & H_3 \end{bmatrix}^{T}=-(\BV_{\bot\bot})^{-1}\BV_{\bot\parallel}\BGy.\label{FNormalComponents}
\end{align}
The fact that the matrix $\BV_{\bot\bot}(z,\omega)$ is invertible follows immediately from the fact that the passivity
properties%
\index{passivity}
(\ref{FPassiveLayeredMedia}) imply
\begin{align}
\operatorname{Im}(\BV_{\bot\bot}(z,\omega))>0.
\end{align}

We will now prove in the next proposition [using the methods developed in the appendix of \citeAPY{Shipman:2013:RES} and in the Ph.D. thesis of \citeAPY{Welters:2011:TSL}], some fundamental properties associated to the ODE (\ref{FMaxwellODEs}). In particular, we will show that the solution of the initial-valued problem for the ODE (\ref{FMaxwellODEs}) depends analytically on the phase moduli.
\begin{Pro}\label{FPropositionOnTransferMatrix}
For each $z_0\in [-d,d]$ (and for fixed $\BGk,\omega$), the initial-value problem for the ODE (\ref{FMaxwellODEs}), i.e.,
\begin{align}
-i\BJ\frac{d\BGy}{dz}&=\BA(z)\BGy(z),\;\;\BGy(z_0)=\BGy_0,\label{ODEsIVP}
\end{align}
has a unique solution $\BGy$ in $(AC([-d,d]))^4$ for each $\BGy_0\in\mathbb{C}^4$ which is given by
\begin{align}
\BGy(z)=\BT(z_0,z)\BGy_0,\;\;z\in[-d,d],
\end{align}
where the $4\times 4$ matrix $\BT(z_0,z)$ is called the \textit{transfer matrix}.%
\index{transfer matrices}
This transfer matrix $\BT$ has the properties
\begin{align}\label{FTransferMatrixProp}
\BT(z_0,z)=\BT(z_1,z)\BT(z_0,z_1),\;\;\BT(z_0,z_1)^{-1}=\BT(z_1,z_0),\;\;\BT(z_0,z_0)=\BI,
\end{align}
for all $z_0, z_1, z\in[-d,d]$. Furthermore, the map
\begin{align}
\BT&=\BT(z_0,z)=\BT(z_0,z,\BGk, \omega)\notag\\
&=\BT(z_0,z, \BGk, \omega \BGve_1(\omega),\omega \BGve_2(\omega), \omega \BGm_1(\omega),\omega \BGm_2(\omega)),\;z_0,z\in[-d,d],\;\BGk\in\bbC^2,\;\omega\in\bbC^+,
\end{align}
belongs to $M_4(AC([-d,d]))$ as a function of $z$ (for fixed $z_0,\BGk,\omega$) and it is an analytic function as a map of $(\BGk,\omega)$ into $M_4(\bbC)$ (for fixed $z_0,z$).  More generally, the map
\begin{align}
\BZ\mapsto \BT(z_0,z, \BGk, \omega \BGve_1,\omega \BGve_2, \omega \BGm_1,\omega \BGm_2)
\end{align}
is analytic as a function of $\BZ=(\omega \BGve_1,\omega \BGve_2, \omega \BGm_1,\omega \BGm_2)\in (M_{3}^{+}(\bbC))^{4}$ into $M_4(\bbC)$ (for fixed $z_0,z,\BGk$).
\end{Pro}
\begin{proof}
First, it follows from Hartogs' Theorem%
\index{Hartogs' theorem}
(see Theorem \ref{FThm.Hartog}), the hypotheses (\ref{DefDielecMagnMatrices1}), (\ref{DefDielecMagnMatrices2}), (\ref{FDielecMagnHerglotz}), the formulas (\ref{FAMatrix})--(\ref{FAMatrixExplicit5}), and Theorem \ref{FTh.analyinvop} that
\begin{align}
(\BGk,\omega)\mapsto \BA(\cdot, \BGk, \omega \BGve_1(\omega),\omega \BGve_2(\omega), \omega \BGm_1(\omega),\omega \BGm_2(\omega))
\end{align}
is analytic as a function into $M_4(L^1(I))$, where $I=[-d,d]$, from $\bbC^2\times \bbC^{+}$. And, more generally, it follows from these theorems, hypotheses, and formulas that the map
\begin{align}
(\BGk,\BZ)\mapsto \BA(\cdot, \BGk, \BZ)
\end{align}
is analytic as a function of $(\BGk,\BZ)\in \bbC^2\times(M_{3}^{+}(\bbC))^{4}$ into $M_4(L^1(I))$, where $\BZ$ is the variable $\BZ=(\omega \BGve_1,\omega \BGve_2, \omega \BGm_1,\omega \BGm_2)$.

In particular, for either fixed variables $(\BGk,\omega)\in\bbC^2\times\bbC^{+}$ or $(\BGk,\BZ)\in \bbC^2\times(M_{3}^{+}(\bbC))^{4}$,
we have $\BA=\BA(z)$ from (\ref{FAMatrix}) is in $M_4(L^1(I))$. Fix a $z_0\in I$. Then by the theory of linear ordinary differential equations [see, for instance,
Theorem 1.2.1 in Chapter 1 of \citeAPY{Zettl:2005:SLT}], the initial-value problem (\ref{ODEsIVP}) has a unique solution $\BGy$ in $(AC(I))^4$
for each $\BGy_0\in\mathbb{C}^4$. Denote the standard orthonormal basis vectors of $\bbR^4$ by $\Bw_j$, for $j=1,2,3,4$. Let $\BGy_j\in (AC(I))^4$
denote the unique solution of the ODE (\ref{FMaxwellODEs}) satisfying $\BGy_j(z_0)=\Bw_j$, for $j=1,2,3,4$. Now
let $\BT(z_0,z)=[\BGy_1(z)|\BGy_2(z)|\BGy_3(z)|\BGy_4(z)]\in M_4(\bbC)$ denote the $4\times 4$ matrix whose columns
are $\BT(z_0,z)\Bw_j=\BGy_j(z)$ for $j=1,2,3,4$ and $z\in I$. This matrix $\BT(z_0,z)$ is known in the
electrodynamics of layered media as the transfer matrix.%
\index{transfer matrices}

Now it follows immediately from the uniqueness of the solution to the initial-value problem (\ref{ODEsIVP}) 
and the definition of the transfer matrix $\BT(z_0,z)$,
that $\BT=\BT(z_0,z)$ as a function of $z\in I$ belongs to $M_4(AC(I))$, it has the properties (\ref{FTransferMatrixProp}), and it is the
unique matrix-valued function in $M_4(AC(I))$ satisfying: if $\BGy_0\in\bbC^4$ then $\BGy(z)=\BT(z_0,z)\BGy_0$ for all $z\in I$ is an $(AC(I))^4$
solution to the initial-value problem (\ref{ODEsIVP}). From this uniqueness property of the transfer matrix $\BT(z_0,z)$,
it follows that $\BT(z_0,z)$ is the unique solution to the initial-value problem:
\begin{align}\label{FTransferMatrixODEs}
\BGY^{\prime}(z)=i\BJ^{-1}\BA(z)\BGY(z),\;\BGY(z_0)=\BI,\;\BGY\in M_4(AC(I)),
\end{align}
where $\BI\in M_4(\bbC)$ is the identity matrix.
 
Now we wish to derive an explicit representation for $\BT(z_0,z)$ in terms of $\BJ$ and $\BA$. To do this we first introduce some results from the
integral operator approach to the theory of linear ODEs. For fixed $\BM\in M_4(L^1(I))$, define the linear map $\BCI[\BM,z_0]:M_4(AC(I))\rightarrow M_4(AC(I))$
by
\begin{align}
(\BCI[\BM,z_0]\BN)(z)=\int_{z_0}^{z}\BM(u)\BN(u)du,\;\BN\in M_4(AC(I)),\;z\in I.
\end{align}
It follows that $\BCI[\BM,z_0]$ is a continuous linear operator on the Banach space $M_4(AC(I))$, i.e., it belongs to $L(M_4(AC(I)),M_4(AC(I)))$.
Next, define the linear map $\BCT[\BM,z_0]:M_4(AC(I))\rightarrow M_4(AC(I))$ by
\begin{align}
\BCT[\BM,z_0]=\mathbf{1}-\BCI[\BM,z_0],
\end{align}
where $\mathbf{1}\in L(M_4(AC(I)),M_4(AC(I)))$ denotes the identity operator on $M_4(AC(I))$. Then it follows that $\BCT[\BM,z_0]\in L(M_4(AC(I)),M_4(AC(I)))$
and, moreover, $\BCT[\BM,z_0]$ is invertible with $\BCT[\BM,z_0]^{-1}\in L(M_4(AC(I)),M_4(AC(I)))$, i.e., $\BCT[\BM,z_0]$ is an isomorphism.
The fact that $\BCT[\BM,z_0]$ is invertible follows immediately from the existence and uniqueness of the solution $\BY\in M_4(AC(I))$
for each $\BC\in M_4(\bbC)$, $\BF\in M_4(L^1(I))$ to the inhomogeneous initial-value problem [see, for instance,
Theorem 1.2.1 in Chapter 1 of \citeAPY{Zettl:2005:SLT}]:
\begin{align}
\BY^{\prime}(z)=\BM(z)\BY(z)+\BF(z),\;\BY(z_0)=\BC.
\end{align}
In other words, $\BY$ is the unique solution in $M_4(AC(I))$ to the integral equation
\begin{align}
\BCT[\BM,z_0]\BY=\BCI[\BM,z_0]\BF+\BC.
\end{align}
Hence, the solution is given explicitly by
\begin{align}
\BY=\BCT[\BM,z_0]^{-1}(\BCI[\BM,z_0]\BF+\BC).
\end{align}
In particular, it follows from this representation and the fact that the transfer matrix $\BT(z_0,z)$ is the unique solution to the initial-value
problem (\ref{FTransferMatrixODEs}) that with $\BF=\mathbf{0}$, $\BC=\BI$ in the notation above,
\begin{align}\label{FTransferMatrixExplicitRep}
\BT(z_0,\cdot)=\BCT[i\BJ^{-1}\BA,z_0]^{-1}(\BI),
\end{align}
where $\BA=\BA(z)$ as you will recall belongs to $M_4(L^1(I))$ as a function of $z\in I$ (ignoring its dependence on the other variables)
and hence so does $i\BJ^{-1}\BA$.

Now since $i\BJ^{-1}\BA$ is an analytic function of either of the variables $(\BGk,\omega)$ or $(\BGk,\BZ)$ into $M_4(L^1(I))$ as a function of $z\in I$,
for fixed $z_0$, then it follows immediately from this, the representation (\ref{FTransferMatrixExplicitRep}), and
Theorem \ref{FTh.analyinvop} that $(\BGk,\omega)\mapsto \BT(z_0,z, \BGk,\omega)$ and $(\BGk,\BZ)\mapsto \BT(z_0,z, \BGk,\BZ)$ are
analytic functions into $M_4(AC(I))$ as a function of $z\in I$, for fixed $z_0$.  Finally, the proof of the rest of this proposition
now follows immediately from these facts and the fact that the Banach space $AC(I)$ can be continuously embedded into the Banach space $C(I)$ of
continuous functions from $I$ into $\bbC$ equipped with the sup norm $||f||_{\infty}=\sup_{z\in I}|f(z)|$, that is, the
identity map $\iota:AC(I)\rightarrow C(I)$ between these two Banach spaces [i.e., $\iota(f)=f$ for $f\in AC(I)$] is a continuous (and hence bounded)
linear map under their respective norms [i.e., $\iota\in L(AC(I),C(I))$].
\end{proof}

\begin{Rem}
Using Proposition \ref{FPropositionOnTransferMatrix} and due to the simplicity of the layered media considered we can derive a simple explicit representation
of the transfer matrix $\BT(z_0,z)$ for all $z_0,z\in [-d,d]$.  First, the transfer matrix $\BT(-d,z)$, $z\in[-d,d]$ takes on the simple form
\begin{align}\label{FTransferMatrixSimpleExpRepr1a}
\BT(-d,z)=\left\{
\begin{array}{cc}
e^{i\BJ\BA_{1}(z+d)}, & -d\leq z\leq -d_{2}, \\
e^{i\BJ\BA_{1}(d-d_{2})}e^{i\BJ\BA_{2}(z+d_{2})}, & -d_{2}\leq z\leq d_{2}, \\
e^{i\BJ\BA_{1}(d-d_{2})}e^{i\BJ\BA_{2}(2d_{2})}e^{i\BJ\BA_{1}(z-d_{2})}, & d_{2}\leq z\leq
d,
\end{array}
\right.
\end{align}
where $\BA_1$ and $\BA_2$ are the matrices (\ref{FAMatrix}) for a $z$-independent homogeneous medium filled with only material 1 (with permittivity and permeability $\BGve_1$ and $\BGm_1$) and with only material 2 (with permittivity and permeability $\BGve_2$ and $\BGm_2$), respectively (see \ref{Ffig:layered}).

Therefore, in terms of this explicit form for $\BT(-d,z)$, it follows from (\ref{FTransferMatrixProp}) that the
transfer matrix $\BT(z_0,z)$, $z_0,z\in [-d,d]$ is given explicitly in terms of (\ref{FTransferMatrixSimpleExpRepr1a}) by
\begin{align}\label{FTransferMatrixSimpleExpRepr1b}
\BT(z_0,z)=\BT(-d,z)\BT(z_0,-d)=\BT(-d,z)\BT(-d,z_0)^{-1}.
\end{align}
\end{Rem}

\subsection{Electromagnetic Dirichlet-to-Neumann Map}\label{subsection2.1}
Now every solution to Maxwell's equations (\ref{FMaxwellEqsTimeHar}) of the form (\ref{FTangentialBlochWaves}) has in terms of its tangential components (\ref{FTangentialComponents}) a corresponding solution of the ODE (\ref{FMaxwellODEs}) with normal components given by (\ref{FNormalComponents}). And conversely, every solution of the ODE (\ref{FMaxwellODEs}) gives the tangential components of a unique solution to equations (\ref{FMaxwellEqsTimeHar}) of the form (\ref{FTangentialBlochWaves}) with normal components expressed
in terms of its tangential components by (\ref{FNormalComponents}). We use this correspondence to now define the electromagnetic ``Dirichlet-to-Neumann''%
\index{Dirichlet-to-Neumann map!electromagnetism}%
\index{DtN map|see{Dirichlet-to-Neumann map}}
(DtN) map in terms of the transfer matrix%
\index{transfer matrices}
$\BT$ whose properties are described in Proposition \ref{FPropositionOnTransferMatrix}.

The DtN map is a function
\begin{gather}
\Lambda=\Lambda(z_0,z_1)=\Lambda(z_0,z_1,\BGk,\omega  \BGve_1(\omega),\omega \BGve_2(\omega),\omega \BGm_1(\omega),\omega \BGm_2(\omega)),\\
 z_0,z_1\in[-d,d],\;z_0<z_1, \BGk\in \mathbb{C}^2,\;\omega\in\bbC^{+},\notag
\end{gather}
which can be defined as the block operator matrix
\begin{gather}
\Lambda(z_0,z_1)\begin{bmatrix}
\textbf{E}\times\textbf{n}|_{z=z_1}\\
\textbf{E}\times\textbf{n}|_{z=z_0}
\end{bmatrix}
=
\begin{bmatrix}
i\textbf{n}\times\textbf{H}\times\textbf{n}|_{z=z_1}\\
i\textbf{n}\times\textbf{H}\times\textbf{n}|_{z=z_0}
\end{bmatrix},
\end{gather}
where $\BE,\BH$ denote a solution of the time-harmonic Maxwell's equations (\ref{FMaxwellEqsTimeHar}) of the form (\ref{FTangentialBlochWaves}), i.e., a function of the form (\ref{FTangentialBlochWaves}) whose tangential components $\BGy$ with the form (\ref{FTangentialComponents}) satisfy the ODE (\ref{FMaxwellODEs}) and whose normal components are given in terms of these tangential components $\BGy$ by (\ref{FNormalComponents}).

A more explicit definition of this DtN map can be given as follows. First, on $\mathbb{C}^3$, with respect to the standard orthonormal basis vectors, we have the matrix representations
\begin{align}\label{FDefe3cross}
\textbf{e}_3\times = \begin{bmatrix}
0 & -1 & 0\\
1 & 0 & 0\\
0 & 0 & 0
\end{bmatrix},\;\;-\textbf{e}_3\times\textbf{e}_3\times = \begin{bmatrix}
1 & 0 & 0\\
0 & 1 & 0\\
0 & 0 & 0
\end{bmatrix},
\end{align}
and this allows us to write $\textbf{E}\times\textbf{n}=-\textbf{n}\times\textbf{E}$ and $\textbf{n}\times\textbf{H}\times\textbf{n}=-\textbf{n}\times\textbf{n}\times\textbf{H}$ as matrix multiplication so that we can write $\Lambda$ as a $6\times 6$ matrix which can be written in the $2\times 2$ block matrix form as
\begin{align}\label{FPreDefLambdaMatrixBlocks}
\Lambda = \begin{bmatrix}
\Lambda_{11} & \Lambda_{12}\\
\Lambda_{21} & \Lambda_{22}
\end{bmatrix}.
\end{align}
We now want to get an explicit expression of this block form. Thus, we define the projections
\begin{align}\label{FDefProjsForDtNMap}
\BP_{\mathfrak{t}}=\begin{bmatrix}
1 & 0  \\
0 & 1 \\
0 & 0 \\
\end{bmatrix},\;\BQ_{\mathfrak{t},1}=\begin{bmatrix}
1 & 0 & 0 & 0 \\
0 & 1 & 0 & 0 \\
\end{bmatrix},\;\;
\BQ_{\mathfrak{t},2}=\begin{bmatrix}
0 & 0 & 1 & 0 \\
0 & 0 & 0 & 1 \\
\end{bmatrix}.
\end{align}
It follows from this notation that
\begin{align*}
\textbf{E}\times\textbf{e}_3&=-e^{i(k_1x+k_2y)}\textbf{e}_3\times \BP_{\mathfrak{t}}\left[\BQ_{\mathfrak{t},1}\BGy(z)\right],\;\;\textbf{n}\times\textbf{H}\times\textbf{n}=e^{i(k_1x+k_2y)}\BP_{\mathfrak{t}}\left[\BQ_{\mathfrak{t},2}\BGy(z)\right].
\end{align*}
Hence, we have
\begin{gather*}
\begin{bmatrix}
i\textbf{n}\times\textbf{H}\times\textbf{n}|_{z=z_1}\\
i\textbf{n}\times\textbf{H}\times\textbf{n}|_{z=z_0}
\end{bmatrix}
=
ie^{i(k_1x+k_2y)}
\begin{bmatrix}
\BP_{\mathfrak{t}} & 0\\
0 & \BP_{\mathfrak{t}}
\end{bmatrix}
\begin{bmatrix}
\BQ_{\mathfrak{t},2}\psi(z_1)\\
\BQ_{\mathfrak{t},2}\psi(z_0)
\end{bmatrix}\\
=
ie^{i(k_1x+k_2y)}
\begin{bmatrix}
\BP_{\mathfrak{t}} & 0\\
0 & \BP_{\mathfrak{t}}
\end{bmatrix}
\BGG(z_0,z_1)\begin{bmatrix}
\BQ_{\mathfrak{t},1}\psi(z_1)\\
\BQ_{\mathfrak{t},1}\psi(z_0)
\end{bmatrix}\\
=
i\begin{bmatrix}
\BP_{\mathfrak{t}} & 0\\
0 & \BP_{\mathfrak{t}}
\end{bmatrix}
\BGG(z_0,z_1)
\begin{bmatrix}
\BP_{\mathfrak{t}} & 0\\
0 & \BP_{\mathfrak{t}}
\end{bmatrix}^{T}
\begin{bmatrix}
\textbf{n}\times\textbf{E}\times\textbf{n}|_{z=z_1}\\
\textbf{n}\times\textbf{E}\times\textbf{n}|_{z=z_0}
\end{bmatrix}\\
=
i\begin{bmatrix}
\BP_{\mathfrak{t}} & 0\\
0 & \BP_{\mathfrak{t}}
\end{bmatrix}
\BGG(z_0,z_1)
\begin{bmatrix}
\BP_{\mathfrak{t}} & 0\\
0 & \BP_{\mathfrak{t}}
\end{bmatrix}^{T}
\begin{bmatrix}
\textbf{e}_3\times & 0\\
0 & -\textbf{e}_3\times\\
\end{bmatrix}
\begin{bmatrix}
\textbf{E}\times\textbf{n}|_{z=z_1}\\
\textbf{E}\times\textbf{n}|_{z=z_0}
\end{bmatrix},
\end{gather*}
where we have used the fact that since
\begin{gather*}
\begin{bmatrix}
\Bu_1\\
\Bv_1
\end{bmatrix}
=
\begin{bmatrix}
\BQ_{\mathfrak{t},1}\BGy(z_1)\\
\BQ_{\mathfrak{t},2}\BGy(z_1)
\end{bmatrix},\;\;
\begin{bmatrix}
\Bu_0\\
\Bv_0
\end{bmatrix}
=
\begin{bmatrix}
\BQ_{\mathfrak{t},1}\BGy(z_0)\\
\BQ_{\mathfrak{t},2}\BGy(z_0)
\end{bmatrix},\;\;\BT(z_0,z_1)\begin{bmatrix}
\Bu_0\\
\Bv_0
\end{bmatrix}
=
\begin{bmatrix}
\Bu_1\\
\Bv_1
\end{bmatrix},
\end{gather*}
then by Proposition \ref{FProbDefinesGamma} (given later in Section \sect{Fmain}) we must have
\begin{gather*}
\BGG(z_0,z_1)\begin{bmatrix}
\BQ_{\mathfrak{t},1}\BGy(z_1)\\
\BQ_{\mathfrak{t},1}\BGy(z_0)
\end{bmatrix}
=
\begin{bmatrix}
\BQ_{\mathfrak{t},2}\BGy(z_1)\\
\BQ_{\mathfrak{t},2}\BGy(z_0)
\end{bmatrix},
\end{gather*}
where $\BGG(z_0,z_1)$ is defined in (\ref{FGammaMatrix}) [which is well-defined provided the matrix $\BT_{12}(z_0,z_1)$ in the block decomposition of $\BT(z_0,z_1)$ in (\ref{DefTransferMatrixBlockDecomp}) is invertible]. Therefore, the DtN map $\Lambda(z_0,z_1)$ can be defined explicitly as follows.

\begin{Def}[Electromagnetic Dirichlet-to-Neumann map]\label{FDefDtNMap}%
\index{Dirichlet-to-Neumann map!electromagnetism}
The electromagnetic DtN map $\Lambda(z_0,z_1)$ is defined to be the $6\times 6$ matrix (\ref{FPreDefLambdaMatrixBlocks}) defined in terms of the $4\times 4$ matrix $\BGG(z_0,z_1)$ in (\ref{FGammaMatrix}) and the $3\times 2$ matrix $\BP_{\mathfrak{t}}$ in (\ref{FDefProjsForDtNMap}) by
\begin{gather}
\Lambda(z_0,z_1)=
i\begin{bmatrix}\label{FGammaToLambda}
\BP_{\mathfrak{t}} & 0\\
0 & \BP_{\mathfrak{t}}
\end{bmatrix}
\BGG(z_0,z_1)
\begin{bmatrix}
\BP_{\mathfrak{t}} & 0\\
0 & \BP_{\mathfrak{t}}
\end{bmatrix}^{T}
\begin{bmatrix}
\mathbf{e}_3\times & 0\\
0 & -\mathbf{e}_3\times\\
\end{bmatrix},
\end{gather}
and in the $2\times 2$ block matrix form its entries are the $3\times 3$ matrices
\begin{align}\label{FGammaToLambdaBlockEntries}
\Lambda_{11}(z_0,z_1)&=i\BP_{\mathfrak{t}}\BGG_{11}(z_0,z_1)\BP_{\mathfrak{t}}^{T}\mathbf{e}_3\times ,\\
\Lambda_{12}(z_0,z_1)&= -i\BP_{\mathfrak{t}}\BGG_{12}(z_0,z_1)\BP_{\mathfrak{t}}^{T}\mathbf{e}_3\times ,\\
\Lambda_{21}(z_0,z_1)&=i\BP_{\mathfrak{t}}\BGG_{21}(z_0,z_1)\BP_{\mathfrak{t}}^{T}\mathbf{e}_3\times ,\\
\Lambda_{22}(z_0,z_1)&=-i\BP_{\mathfrak{t}}\BGG_{22}(z_0,z_1)\BP_{\mathfrak{t}}^{T}\mathbf{e}_3\times,
\end{align}
where $\mathbf{e}_3\times $ is the $3\times 3$ matrix in (\ref{FDefe3cross}).
\end{Def}

Now for any $z_0,z_1\in [-d,d]$, $z_0<z_1$, we want to know whether the DtN map $\Lambda(z_0,z_1)$ is well-defined or not. The next theorem addresses this.
\begin{Thm}\label{FThm:LayeredDtnMapWellDef}
If $\,\operatorname{Im}\omega>0$ and $\BGk\in \mathbb{R}^2$ then for any $3\times 3$ matrix-valued Herglotz functions%
\index{Herglotz functions!matrix-valued}
$\omega \BGve_j(\omega)$, $\omega \BGm_j(\omega)$, $j=1,2$ with range in $M_3^+(\mathbb{C})$, the electromagnetic DtN map $\Lambda(z_0,z_1,\BGk,\omega  \BGve_1(\omega),\omega \BGve_2(\omega),\omega \BGm_1(\omega),\omega \BGm_2(\omega))$ is well-defined.
\end{Thm}
\begin{proof}
Let $\omega \BGve_j(\omega)$, $\omega \BGm_j(\omega)$, $j=1,2$ be any $3\times 3$ matrix-valued Herglotz functions with range in $M_3^+(\mathbb{C})$. Choose any values $\omega\in\bbC$ and $\BGk$ with $\Imag{\omega}>0$ and $\BGk\in\bbR^2$. Consider the time-harmonic Maxwell's equations (\ref{FMaxwellEqsTimeHar}) for the plane parallel layered media in \ref{Ffig:layered} at the frequency $\omega$ for solutions of the form (\ref{FTangentialBlochWaves}) with tangential wavevector $\BGk$, where the dielectric permittivity $\BGve$ and magnetic permeability $\BGm$ are defined in (\ref{DefDielecMagnMatrices1}) and (\ref{DefDielecMagnMatrices2}).

For $z_0,z_1\in[-d,d]$ with $z_0<z_1$, the transfer matrix%
\index{transfer matrices}
(defined in Section \ref{SecFormProblemLayered}) of the layered media with tensors $\BGve(z,\omega)$, $\BGm(z,\omega)$ is $\BT(z_0,z_1,\BGk,\omega  \BGve_1(\omega),\omega \BGve_2(\omega),\omega \BGm_1(\omega),\omega \BGm_2(\omega))$. For simplicity we will suppress the dependency on the other parameters and denote this transfer matrix by $\BT(z_0,z_1)$. It now follows from the passivity property (\ref{FPassiveLayeredMedia}) and Theorem \ref{FThmPassivityImpliesDtNMapDefined}, given below, that the matrix $\BJ-\BT(z_0,z_1)^*\BJ\BT(z_0,z_1)$ is positive definite. By Proposition \ref{FProp1stKeyDtNMapDefined}, given below, it follows that the $2\times 2$ matrices $\BT_{ij}(z_0,z_1)$, $1\leq i,j\leq 2$, that make up the blocks for the transfer matrix $\BT(z_0,z_1)$ in the $2\times 2$ block form in (\ref{FTijMatrices}), are invertible. It follows from this that the matrix $\BGG(z_0,z_1)$ defined in (\ref{FGammaMatrix}) terms of these $2\times 2$ matrices is well-defined. And therefore it follows from the fact that $\BGG(z_0,z_1)$ is well-defined that the electromagnetic DtN map $\Lambda(z_0,z_1)=\Lambda(z_0,z_1,\BGk,\omega  \BGve_1(\omega),\omega \BGve_2(\omega),\omega \BGm_1(\omega),\omega \BGm_2(\omega))$, as given in Definition \ref{FDefDtNMap}, is well-defined. This completes the proof.
\end{proof}

The main result of this section on the analytic properties of the DtN map is the following:
\begin{Thm}
For any $\BGk\in \mathbb{R}^2$ and any $3\times 3$ matrix-valued Herglotz functions%
\index{Herglotz functions!matrix-valued}
$\omega \BGve_j(\omega)$, $\omega \BGm_j(\omega)$, $j=1,2$ with range in $M_3^+(\mathbb{C})$, the function
\begin{align}
\omega\mapsto \Lambda(z_0,z_1,\BGk,\omega \BGve_1(\omega),\omega \BGve_2(\omega),\omega \BGm_1(\omega),\omega \BGm_2(\omega))
\end{align}
is analytic from $\mathbb{C}^+$ into $M_6^+(\mathbb{C})$ and, in particular, it is a matrix-valued Herglotz function. More generally, it is a 
Herglotz function in the variable $\BZ=(\omega \BGve_1,\omega \BGve_2, \omega \BGm_1,\omega \BGm_2)\in (M_3^+(\mathbb{C}))^4$ 
(see Definition \ref{FDef.Herg1}).
\end{Thm}
\begin{proof}
Fix any $3\times 3$ matrix-valued Herglotz functions
$\omega \BGve_j(\omega)$, $\omega \BGm_j(\omega)$, $j=1,2$ with range in $M_3^+(\mathbb{C})$. Then for any electromagnetic field $\mathbf{E}$\textbf{, }$\mathbf{B}$\textbf{\ }with
tangential components $\BGy$ with $\operatorname{Im}\omega>0$ and tangential wavevector $\BGk\in\mathbb{R}^2$ we have, by Theorem \ref{FThmPassivityImpliesDtNMapDefined} and its proof, that
\begin{gather*}
\left(
\begin{bmatrix}
\mathbf{E}\times\mathbf{n}|_{z=z_{1}} \\
\mathbf{E}\times\mathbf{n}|_{z=z_{0}}
\end{bmatrix}
,\operatorname{Im}\left[\Lambda \left( z_{0},z_{1}\right)\right]
\begin{bmatrix}
\mathbf{E}\times\mathbf{n}|_{z=z_{1}} \\
\mathbf{E}\times\mathbf{n}|_{z=z_{0}}
\end{bmatrix}
\right)=\operatorname{Re}\left(
\begin{bmatrix}
\mathbf{E}\times\mathbf{n}|_{z=z_{1}} \\
\mathbf{E}\times\mathbf{n}|_{z=z_{0}}
\end{bmatrix}
,
\begin{bmatrix}
\mathbf{n}\times\mathbf{H}\times\mathbf{n}|_{z=z_{1}} \\
\mathbf{n}\times\mathbf{H}\times\mathbf{n}|_{z=z_{0}}
\end{bmatrix}
\right)\\
=\operatorname{Re}\left\{ \left( \mathbf{E}\times\mathbf{n}|_{z=z_{1}}, \mathbf{n}\times\mathbf{H}\times\mathbf{n}|_{z=z_{1}}\right) +\left( \mathbf{E}\times\mathbf{n}|_{z=z_{0}}, \mathbf{n}\times\mathbf{H}\times\mathbf{n}|_{z=z_{0}}\right) \right\}  \\
=-\frac{1}{2}\left( \BGy \left( z_{1}\right) ,\BJ\BGy \left( z_{1}\right)
\right) +\frac{1}{2}\left( \BGy \left( z_{0}\right) ,\BJ\BGy \left(
z_{0}\right) \right)
\\
=\frac{1}{c}\int\limits_{z_{0}}^{z_{1}}\left( \mathbf{H},\operatorname{Im}\left[\omega \BGm \left(
z,\omega \right)\right] \mathbf{H}\right) +\left( \mathbf{E},\operatorname{Im}\left[\omega \BGve \left( z,\omega \right)\right] \mathbf{E}\right) dz\geq 0,
\end{gather*}
with equality if and only if $\BGy\equiv 0$. It now follows from this and Theorem \ref{FThmPassivityImpliesDtNMapDefined}, which tells us that $\BJ-\BT(z_0,z_1)^*\BJ\BT(z_0,z_1)$ is positive definite, that we must have $\operatorname{Im}{\Lambda(z_0,z_1)}>0$.

We will now prove that the function $\omega\mapsto \Lambda(z_0,z_1,\BGk,\omega \BGve_1(\omega),\omega \BGve_2(\omega),\omega \BGm_1(\omega),\omega \BGm_2(\omega))$ is analytic from $\mathbb{C}^+$ into $M_6^+(\mathbb{C})$. By Proposition \ref{FPropositionOnTransferMatrix} we know that the map $$\omega\mapsto \BT(z_0,z_1,\BGk,\omega \BGve_1(\omega),\omega \BGve_2(\omega),\omega \BGm_1(\omega),\omega \BGm_2(\omega))$$ is an analytic function into $M_4(\mathbb{C})$. This implies by (\ref{FGammaMatrix}), (\ref{FGammaMatrixExplicit}) and Theorem \ref{Th.opanalytic} that $$\omega\mapsto \BGG(z_0,z_1,\BGk,\omega \BGve_1(\omega),\omega \BGve_2(\omega),\omega \BGm_1(\omega),\omega \BGm_2(\omega))$$ is an analytic function into $M_4(\mathbb{C})$ and so by (\ref{FGammaToLambda}) it follows that $$\omega\mapsto \Lambda(z_0,z_1,\BGk,\omega \BGve_1(\omega),\omega \BGve_2(\omega),\omega \BGm_1(\omega),\omega \BGm_2(\omega))$$ is an analytic function into $M_6^+(\mathbb{C})$.

Now we introduce the variable $\BZ=(\omega \BGve_1, \omega \BGve_2, \omega \BGm_1, \omega \BGm_2)\in (M_3^+(\mathbb{C}))^4$. Here $M_3^+(\mathbb{C})$ is an open, connected, and convex subset of $M_3(\mathbb{C})$ as a Banach space%
\index{Banach space}
in any normed topology (as all norms on a finite-dimensional vector space are equivalent) and hence so is $(M_3^+(\mathbb{C}))^4$ as a subset of $(M_3(\mathbb{C}))^4$. Our goal is to prove that the function $\BZ\mapsto \Lambda(z_0,z_1,\BGk,\BZ)$ is analytic. Now as $(M_3(\mathbb{C}))^4$ equipped with any norm is a Banach space and is isomorphic to the Banach space $\mathbb{C}^{36}$ (by mapping the components of the $4$-tuple and their matrix entries to a $36$-tuple) equipped with standard inner product on $\mathbb{C}^{36}$. Thus, by Theorem \ref{FThm.Hartog} (Hartogs' Theorem)%
\index{Hartogs' theorem}
it suffices to prove that for each component $Z_j$ of $\BZ$ as an element of $\mathbb{C}^{36}$, the function $Z_j\mapsto \Lambda(z_0,z_1,\BGk,\BZ)$ is analytic for all other components of $\BZ\in (M_3^+(\mathbb{C}))^4$ fixed. But this proof follows exactly as we did for proving $\omega\mapsto \Lambda(z_0,z_1,\BGk,\omega \BGve_1(\omega),\omega \BGve_2(\omega),\omega \BGm_1(\omega),\omega \BGm_2(\omega))$ is an analytic function into $M_6^+(\mathbb{C})$. Therefore, $\BZ\mapsto \Lambda(z_0,z_1,\BGk,\BZ)$ is analytic. This completes the proof.
\end{proof}
\subsection{Auxiliary results}
\labsect{Fmain}
In this section we will derive some auxiliary results that are used in the preceding subsection. First, we write the transfer matrix $\BT(z_0,z_1)$ in the $2\times 2$ block matrix form
\begin{align}\label{FTijMatrices}
\BT=\begin{bmatrix}
\BT_{11} & \BT_{12}\\
\BT_{21} & \BT_{22}
\end{bmatrix}
\end{align}
with respect to the decomposition $\mathbb{C}^4=\mathbb{C}^2\oplus\mathbb{C}^2$. We next define the $4\times 4$ matrix $\BGG(z_0,z_1)$ in the $2\times 2$ block matrix form by
\begin{gather}
\BGG(z_0,z_1) = \begin{bmatrix}
\BGG_{11}(z_0,z_1) & \BGG_{12}(z_0,z_1)\\
\BGG_{21}(z_0,z_1) & \BGG_{22}(z_0,z_1)
\end{bmatrix}\label{FGammaMatrix}\\
=
\begin{bmatrix}
\BT_{22}(z_0,z_1)\BT_{12}(z_0,z_1)^{-1} & \BT_{21}(z_0,z_1)-\BT_{22}(z_0,z_1)\BT_{12}(z_0,z_1)^{-1}\BT_{11}(z_0,z_1)\\
\BT_{12}(z_0,z_1)^{-1} & -\BT_{12}(z_0,z_1)^{-1}\BT_{11}(z_0,z_1)
\end{bmatrix}, \label{FGammaMatrixExplicit}
\end{gather}
provided $\BT_{12}(z_0,z_1)$ is invertible.

Let us now give an overview of the purpose of the results in this section. Using the next proposition, Proposition \ref{FProbDefinesGamma}, we 
are able to give an explicit formula for the DtN map $\Lambda(z_0,z_1)$ in terms of the transfer matrix $\BT(z_0,z_1)$ using the 
matrix $\BGG(z_0,z_1)$, the latter of which requires the invertibility of the matrix $\BT_{12}(z_0,z_1)$. The proposition which follows after this one, i.e., Proposition \ref{FProp1stKeyDtNMapDefined}, then tells us that the matrix $\BT_{12}(z_0,z_1)$ is invertible, provided the matrix $\BJ-\BT(z_0,z_1)^*\BJ\BT(z_0,z_1)$ is positive definite. And, finally, Theorem \ref{FThmPassivityImpliesDtNMapDefined} tells us that this matrix is positive definite (due to passivity).%
\index{passivity}

\begin{Pro}\label{FProbDefinesGamma}
If $\,\BT_{12}(z_0,z_1)$ is invertible then for any $\Bu_0,\Bu_1\in\mathbb{C}^2$ there exist unique $\Bv_0,\Bv_1\in \mathbb{C}^2$ satisfying
\begin{align}\label{DefTransferMatrixBlockDecomp}
\BT(z_0,z_1)\begin{bmatrix}
\Bu_0 \\
\Bv_0
\end{bmatrix}
=
\begin{bmatrix}
\Bu_1 \\
\Bv_1
\end{bmatrix}.
\end{align}
These unique vectors $\Bv_0,\Bv_1$ are given explicitly in terms of the vectors $\Bu_0,\Bu_1$ by the formula
\begin{align}
\begin{bmatrix}
\Bv_1 \\
\Bv_0
\end{bmatrix}
=\BGG(z_0,z_1)\begin{bmatrix}
\Bu_1 \\
\Bu_0
\end{bmatrix}.
\end{align}
\end{Pro}
\begin{proof}
Assume $\BT_{12}(z_0,z_1)$ is invertible. Let $\Bu_0,\Bu_1\in\mathbb{C}^2$. Then we have
\begin{align*}
\begin{bmatrix}
\Bu_1 \\
\Bv_1
\end{bmatrix}
=\BT(z_0,z_1)\begin{bmatrix}
\Bu_0 \\
\Bv_0
\end{bmatrix}
=\begin{bmatrix}
\BT_{11}(z_0,z_1)\Bu_0+\BT_{12}(z_0,z_1)\Bv_0  \\
\BT_{21}(z_0,z_1)\Bu_0+\BT_{22}(z_0,z_1)\Bv_0
\end{bmatrix}
\end{align*}
if and only if
\begin{align*}
\begin{bmatrix}
0 & \BI \\
\BI & 0
\end{bmatrix}
\begin{bmatrix}
\BI & -\BT_{22}(z_0,z_1) \\
0 & \BT_{12}(z_0,z_1)
\end{bmatrix}
\begin{bmatrix}
\Bv_1\\
\Bv_0
\end{bmatrix}
=
\begin{bmatrix}
\BI & -\BT_{11}(z_0,z_1) \\
0 & \BT_{21}(z_0,z_1)
\end{bmatrix}
\begin{bmatrix}
\Bu_1 \\
\Bu_0
\end{bmatrix},
\end{align*}
and this holds if and only if
\begin{gather*}
\begin{bmatrix}
\Bv_1 \\
\Bv_0
\end{bmatrix}
=
\begin{bmatrix}
\BI & \BT_{22}(z_0,z_1)\BT_{12}(z_0,z_1)^{-1} \\
0 & \BT_{12}(z_0,z_1)^{-1}
\end{bmatrix}
\begin{bmatrix}
0 & \BI \\
\BI & 0
\end{bmatrix}
\begin{bmatrix}
\BI & -\BT_{11}(z_0,z_1) \\
0 & \BT_{21}(z_0,z_1)
\end{bmatrix}
\begin{bmatrix}
\Bu_1 \\
\Bu_0
\end{bmatrix}\\
=
\begin{bmatrix}
\BT_{22}(z_0,z_1)\BT_{12}(z_0,z_1)^{-1} & \BT_{21}(z_0,z_1)-\BT_{22}(z_0,z_1)\BT_{12}(z_0,z_1)^{-1}\BT_{11}(z_0,z_1)\\
\BT_{12}(z_0,z_1)^{-1} & -\BT_{12}(z_0,z_1)^{-1}\BT_{11}(z_0,z_1)
\end{bmatrix}
\begin{bmatrix}
\Bu_1 \\
\Bu_0
\end{bmatrix}.
\end{gather*}
The proof of this proposition follows immediately from these equivalent statements.
\end{proof}

\begin{Pro}\label{FProp1stKeyDtNMapDefined}
The matrix $\BJ-\BT^*\BJ\BT$ [dropping dependency on $(z_0,z_1)$ for simplicity] has the block form
\begin{equation}\label{FJminusTastJTMatrix}
\BJ-\BT^{\ast }\BJ\BT=
\begin{bmatrix}
2\operatorname{Re}\left( \BT_{11}^{\ast }\BGr \BT_{21}\right)  & \BGr -\left(
\BT_{21}^{\ast }\BGr ^{\ast }\BT_{12}+\BT_{11}^{\ast }\BGr \BT_{22}\right)  \\
\left[ \BGr -\left(\BT_{21}^{\ast }\BGr ^{\ast }\BT_{12}+\BT_{11}^{\ast }\BGr
\BT_{22}\right)\right] ^{\ast } & 2\operatorname{Re}\left( \BT_{12}^{\ast }\BGr \BT_{22}\right)
\end{bmatrix},
\end{equation}
where $\operatorname{Re}(\BM)=\frac{1}{2}(\BM+\BM^*)$ denotes the real part of a square matrix $\BM$. In particular, if $\BJ-\BT^{\ast }\BJ\BT>0$ then $\operatorname{Re}\left( \BT_{11}^{\ast }\BGr \BT_{21}\right)>0$, $\operatorname{Re}\left( \BT_{12}^{\ast }\BGr \BT_{22}\right)>0$, and $\BT_{ij}$ is invertible for $1\leq i,j\leq 2$.
\end{Pro}
\begin{proof}
The block representation (\ref{FJminusTastJTMatrix}) follows immediately from the block representations (\ref{FJRhoMatrix}), (\ref{FTijMatrices}) by block multiplication. Suppose $\BJ-\BT^{\ast }\BJ\BT>0$. Then it follows immediately from the block representation (\ref{FJminusTastJTMatrix}) that $\operatorname{Re}\left( \BT_{11}^{\ast }\BGr \BT_{21}\right)>0$, $\operatorname{Re}\left( \BT_{12}^{\ast }\BGr \BT_{22}\right)>0$. Now it is a well-known fact from linear algebra that if $\operatorname{Re}\BM>0$ for a square matrix $\BM$ then $\BM$ is invertible. From this it immediately follows that $\BT_{ij}$ is invertible for $1\leq i,j\leq 2$. This completes the proof.
\end{proof}

Now we define the indefinite inner product $\left[\cdot,\cdot\right]:\mathbb{C}^4\times\mathbb{C}^4\rightarrow \mathbb{C}$ in terms of the standard inner product $\left(\cdot,\cdot\right):\mathbb{C}^4\times\mathbb{C}^4\rightarrow \mathbb{C}$ by
\begin{align}
\left[\BGy_1,\BGy_2 \right] = \frac{c}{16\pi}\left(\BJ\BGy_1,\BGy_2\right),\;\;\BGy_1,\BGy_2\in\mathbb{C}^4.
\end{align}
We also define the complex Poynting vector%
\index{Poynting vector}
$\mathbf{S}$ for functions of the form (\ref{FTangentialBlochWaves}) to be
\begin{gather*}
\mathbf{S}=\frac{c}{8\pi}\mathbf{E\times}\overline{\mathbf{H}}=e^{-2(\operatorname{Im}(k_1)x+\operatorname{Im}(k_2)y)}\mathbf{S}\left( z\right),\;\;
\mathbf{S}\left(z\right)=\frac{c}{%
8\pi }\mathbf{E}\left( z\right) \mathbf{\times }\overline{\mathbf{H}\left(
z\right)}
\end{gather*}
The energy conservation law for Maxwell's equations (\ref{FMaxwellEqsTimeHar}) for functions of the form (\ref{FTangentialBlochWaves}) is now described by the next theorem.
\begin{Thm}\label{FThmPassivityImpliesDtNMapDefined}
Assume $\operatorname{Im}\omega>0$ and $\BGk\in \mathbb{R}^2$. Then for any $z_0, z_1\in [-d,d]$, $z_0<z_1$ and any solution $\BGy$ of the ODE (\ref{FMaxwellODEs}) with $\begin{bmatrix} \bf{E} & \bf{H} \end{bmatrix}^{T}$ the corresponding solution of Maxwell's equations (\ref{FMaxwellEqsTimeHar}) of the form (\ref{FTangentialBlochWaves}) whose tangential components (\ref{FTangentialComponents}) are $\BGy$, we have
\begin{gather}
[\BGy(z_0),\BGy(z_0)]-[\BGy(z_1),\BGy(z_1)]=-\int\limits_{z_{0}}^{z_{1}}\partial _{z}\left[ \operatorname{Re}\mathbf{S}\left(z\right) \cdot e_{3}\right] dz
=
-\int\limits_{z_{0}}^{z_{1}}\nabla \cdot
\operatorname{Re}\left( \mathbf{S}\right) dz \label{FFluxFormToPoynting1}\\
=
\frac{1}{8\pi }\int\limits_{z_{0}}^{z_{1}}\left( \mathbf{H,}\operatorname{Im}\left[ \omega \BGm \left( z,\omega \right) \right] \mathbf{H}\right)
+\left( \mathbf{E,}\operatorname{Im}\left[ \omega \BGve \left( z,\omega
\right) \right] \mathbf{E}\right)dz\geq 0, \label{FFluxFormToPoynting2}
\end{gather}
with equality if and only if $\BGy\equiv 0$. In particular, this implies
\begin{eqnarray}
\BJ-\BT(z_0,z_1,\BGk, \omega)^*\BJ\BT(z_0,z_1, \BGk, \omega)>0.\label{FFluxFormToPoynting3}
\end{eqnarray}
\end{Thm}
\begin{proof}
The equalities in (\ref{FFluxFormToPoynting1}) follow immediately from the equalities
\begin{align*}
\operatorname{Re}\mathbf{S}\left( z\right) \cdot e_{3} =-\frac{1%
}{2}\left(
\begin{bmatrix}
\mathbf{E}\left( z\right)  \\
\mathbf{H}\left( z\right)
\end{bmatrix}%
,%
\begin{bmatrix}
0 & \mathbf{e}_{3}\times  \\
-\mathbf{e}_{3}\times  & 0%
\end{bmatrix}%
\begin{bmatrix}
\mathbf{E}\left( z\right)  \\
\mathbf{H}\left( z\right)
\end{bmatrix}%
\right)
=\frac{1}{2}\left( \BGy \left( z\right) ,\BJ\BGy \left( z\right) \right).
\end{align*}
The proof of the last term in
(\ref{FFluxFormToPoynting1}) being equal to (\ref{FFluxFormToPoynting2}) is proved in almost the exact same way as the proof of Poynting's Theorem for time-harmonic fields [see Section 6.8
in \citeAPY{Jackson:1999:CE} and also Section V.A of \citeAPY{Welters:2014:SLL}] and so will be omitted. The inequality in (\ref{FFluxFormToPoynting2}) follows from passivity (\ref{FPassiveLayeredMedia}) and necessary and sufficient conditions for equality follow immediately from this. These facts imply immediately the inequality in (\ref{FFluxFormToPoynting3}). This completes the proof.
\end{proof}

\section{Analyticity of the DtN map for bounded media}\label{Sec.boundedmedia}
 \setcounter{equation}{0}
\subsection{Formulation of the problem}
For the sake of simplicity, we consider here an electromagnetic medium (see \ref{fig.mediumbw} for an example) composed of two isotropic homogeneous materials which fills an open connected bounded Lipschitz domain $\Omega \subset \bbR^3$ (we refer to the Section 5.1 of \citeAY{Kirsch:2015:MTT} for the definition of Lipschitz bounded domains%
\index{Lipschitz bounded domains}
which includes domains with nonsmooth boundary as polyhedra).
However, our result could be easily extended to a medium composed of a finite number of anisotropic homogeneous materials, this is discussed in the last section. Thus, the dielectric permittivity $\ep$ and the magnetic permeability $\mu$ which characterized this medium are supposed to be piecewise constant functions which take respectively the complex values $\ep_1$ and $\mu_1$ in the first material, and $\ep_2$ and $\mu_2$ in the second one.
Moreover, we assume that both materials are passive, thus these functions have to satisfy (see \citeAY{Milton:2002:TOC}; \citeAY{Welters:2014:SLL}; \citeAY{Bernland:2011:SRC};
\citeAY{Gustafsson:2010:SRP}):
\begin{equation}\label{Feq.pass}
\operatorname{Im}(\omega \ep)>0 \, \mbox{ and } \, \operatorname{Im}(\omega \mu)>0 \mbox{ for } \, \operatorname{Im}\omega>0,
\end{equation}
where $\omega$ denotes the complex frequency. \noindent

The time-harmonic Maxwell's equations%
\index{Maxwell's equations}
(in Gaussian units) which link the electric and magnetic fields $\BE$ and $\BH$ in $\Omega$ are given by:
\begin{equation*}\label{FTM}
(\mathcal{P})\left\{ \begin{array}{ll}
\Curlop \BE-i  \omega \mu \, c^{-1} \BH= 0& \mbox{in }  \Omega, \\[10pt]
\Curlop \BH +i  \omega \ep \, c^{-1}  \BE= 0& \mbox{in }  \Omega, \\[10pt]
\BE \times \Bn=\Bf &  \mbox{on }  \partial \Omega.
\end{array} \right.
\end{equation*}
where $\Bn$ denotes here the outward normal vector on the boundary of $\Omega$: $\partial\Omega$, $c$ the speed of light in the vacuum and $\Bf$ the tangential electric field $\BE$ on $\partial \Omega$.\\

\begin{figure}[t]
\centering
\includegraphics[width=0.69\textwidth]{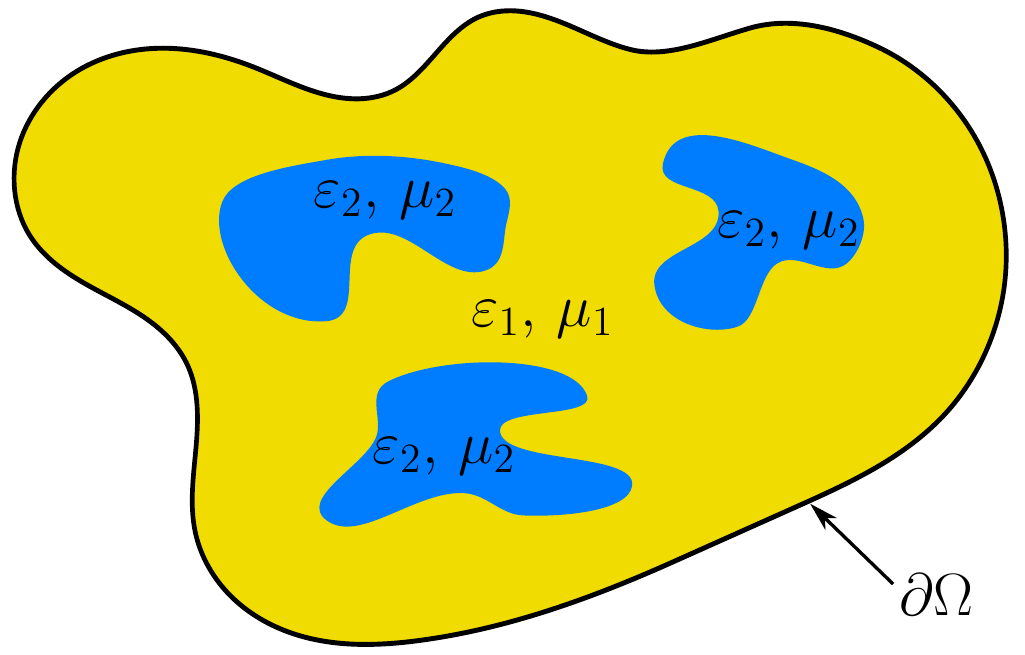}
\caption{Example of the body $\Omega$.}
\label{fig.mediumbw}
\end{figure}

\noindent Let us first introduce some classical functional spaces associated to the study of Maxwell's equations:
\begin{itemize}
\item $L^2(\Omega)$ which is a Hilbert space endowed with the inner product:
$$
(\Bf,\Bg)_{L^2(\Omega)}=\int_{\Omega} \Bf(x)\cdot  \overline{\Bg(x)} \, \dint \Bx,
$$
\item $H(\Curlop,\Omega)=\left\{ \Bu\in L^2(\Omega) \,\mid \,\Curlop \Bu \in L^2(\Omega)\right\}$,
\item $H_{0}(\Curlop,\Omega)=\left\{ \Bu\in H(\Curlop,\Omega)\, \mid \, \Bu \times \Bn=0 \mbox{ on } \partial \Omega \right\}$,
\item $H^{-\frac{1}{2}}(\Divop, \partial\Omega)=\left\{ (\Bu \times \Bn)_{\partial \Omega} \,  \mid \, \Bu \in H(\Curlop,\Omega) \right\}$,
\item $H^{-\frac{1}{2}}(\Curlop, \partial\Omega)=\left\{ \Bn\times (\Bu\times \Bn)_{\partial \Omega} \,  \mid \,\Bu \in H(\Curlop,\Omega)\right\}. $
\end{itemize}
Here $H(\Curlop,\Omega)$ and $H_{0}(\Curlop,\Omega)$ are Hilbert spaces endowed with the norm $\left\|\cdot\right\|_{H(\Curlop,\Omega)}$ defined by
\begin{equation*}
\left\|\Bu\right\|_{H(\Curlop,\Omega)}^2=\left\|\Bu\right\|_{L^2(\Omega)}^2+\left\|\Curlop \Bu\right\|_{L^2(\Omega)}^2.
\end{equation*}
Concerning the functional framework associated with the spaces of tangential traces and tangential trace components $H^{-\frac{1}{2}}(\Divop,\partial\Omega)$ and $H^{-\frac{1}{2}}(\Curlop, \partial\Omega)$, we refer to the Section 5.1 of \citeAPY{Kirsch:2015:MTT}. These spaces are respectively Banach spaces for the norms $\| \cdot \|_{H^{-\frac{1}{2}}(\Curlop,\partial\Omega)}$ and $\| \cdot \|_{H^{-\frac{1}{2}}(\Divop,\partial\Omega)}$ introduced in the Definition 5.23 of \citeAPY{Kirsch:2015:MTT} and are linked by the duality relation: $(H^{-\frac{1}{2}}(\Divop,\partial\Omega))^*=H^{-\frac{1}{2}}(\Curlop,\partial\Omega)$. Moreover, their duality product $\left\langle \cdot , \cdot \right\rangle$ (see Theorem 5.26 of \citeAPY{Kirsch:2015:MTT}) satisfies the following Green's identity:
\begin{equation}\label{Feq.duality}
\int_{\Omega} \Bu\cdot  \Curlop \Bv - \Bv \cdot  \Curlop \Bu \, \dint \Bx=\left\langle\Bn\times(\Bv\times \Bn)  , \Bu\times \Bn \right\rangle, \ \forall \Bu, \Bv \in H(\Curlop,\Omega).
\end{equation}

Here we look for solutions $(\BE, \BH)\in H(\Curlop,\Omega)^2$ of the problem $(\mathcal{P})$ for data $\Bf \in H^{-\frac{1}{2}}(\Divop,\partial\Omega)$.
\subsection{The Dirichlet-to-Neumann map}
We introduce the variable $\BZ=(\omega \ep_1, \omega \ep_2, \omega \mu_1, \omega \mu_2)\in (\bbC^{+})^{4}$.
The electromagnetic Dirichlet-to-Neumann map%
\index{Dirichlet-to-Neumann map!electromagnetism}
$\Lambda_{\BZ}:H^{-\frac{1}{2}}(\Divop,\partial\Omega) \to  H^{-\frac{1}{2}}(\Curlop,\partial\Omega)$ associated to the problem $(\mathcal{P})$ is defined as the linear operator:
\begin{equation}\label{Feq.DtN}
\Lambda_{\BZ}\, \Bf=i\,  \Bn \times (\BH \times \Bn)_{\partial \Omega},  \quad \forall  \Bf \in H^{-\frac{1}{2}}(\Divop,\partial\Omega).
\end{equation}
\begin{Rem}
This definition of the DtN map (\ref{Feq.DtN}) is slightly different from the one introduced in \citeAPY{Albanese:2006:ISP}, \citeAPY{Ola:1993:IBV} and \citeAPY{Uhlmann:2014:REO}.
Here, the rotated tangential electric field $\Bf= \BE \times \Bn$ is mapped (up to a constant) to the tangential component of the magnetic field $\Bn \times (\BH \times \Bn)=(\BI-\Bn\Bn^{T})\BH$
and not to the rotated tangential magnetic field $\BH \times \Bn$. This definition is closer to the one used in \citeAPY{Chaulet:2014:ESP} to construct generalized impedance boundary conditions%
\index{electromagnetism!generalized impedance boundary conditions}
for electromagnetic scattering problems.
\end{Rem}
We want to prove the following theorem:
\begin{Thm}\label{FTh.DtN}
The DtN map $\Lambda_{\BZ}$ is well-defined, is a continuous linear operator with respect to the datum $\Bf$ and is an analytic function of $\BZ$ in the open set $(\bbC^{+})^{4}$. Moreover, the operator $\Lambda_{\BZ}$ satisfies
\begin{equation}\label{Feq.pos}
\operatorname{Im }\left\langle\Lambda_{\BZ}\, \Bf, \overline{\Bf} \right\rangle > 0, \ \forall \Bf \in  H^{-\frac{1}{2}}(\Divop,\partial\Omega)-\{0\},
\end{equation}
and as an immediate consequence, the function
\begin{equation}\label{Feq.herg} h_{\Bf}(\BZ)=   \left\langle  \Lambda_{\BZ} \,\Bf, \overline{\Bf} \right\rangle \mbox{ defined on } (\bbC^{+})^{4} \mbox{ for all $\Bf\in H^{-\frac{1}{2}}(\Divop,\partial\Omega)$}\end{equation}
is a Herglotz function%
\index{Herglotz functions}
of $\BZ$ (see Definition \ref{FDef.Herg1}). 
\end{Thm}

\begin{Rem}
A similar theorem is obtained in the previous chapter of this book for a DtN map defined as the operator which maps the tangential electric field $\Bn \times (\BE \times \Bn)$ to $i\, \Bn \times \BH$ on $\partial \Omega$ . But for a regular boundary $\partial \Omega$ (for example $C^{1,1}$), this other definition of the DtN map can be rewritten as $-Q  \Lambda_{\BZ} \, Q$ with the isomorphism  $Q: H^{-\frac{1}{2}}(\Curlop,\partial\Omega)\to H^{-\frac{1}{2}}(\Divop,\partial\Omega) $ defined by
$
Q(\Bg)=- \Bn \times \Bg$.  Thus, one can show in the same way that the function $$h_{\Bg} (\BZ)=\left\langle \overline{\Bg}  , -Q  \Lambda_{\BZ}\, Q\,\Bg \right\rangle, \, \forall \Bg\in H^{-\frac{1}{2}}(\Curlop,\partial\Omega) $$ is a Herglotz function on $(\bbC^{+})^{4}$.
But, as it is mentioned in Remark 1, p.\ 30 and Corollary 2, p.\ 38 of \citeAPY{Cessenat:1996:MME}, the isomorphism $Q$ may not be well-defined if the function $\Bn$ is not regular enough.
That is why, in order to make this connection, we assume that the boundary $\partial \Omega$ is slightly more regular than Lipschitz continuous.
\end{Rem}
\subsection{Proof of the Theorem \ref{FTh.DtN}}\label{sec.proof}

We will first prove that the linear operator $T_{\BZ}:H^{-\frac{1}{2}}(\Divop,\partial\Omega)\to H(\Curlop,\Omega)^2 $ which associates the data $\Bf$ to the solution  $(\BE, \BH)\in H(\Curlop,\Omega)^2$ of $(\mathcal{P})$ is well-defined, continuous and analytic in $\BZ$ in $(\bbC^{+})^{4}$. In other words that the problem $(\mathcal{P})$ admits a unique solution $(\BE, \BH)$ which depends continuously on the data $\Bf$ and analytically on $\BZ$. The approach we follow is standard, it uses a variational reformulation of the time-harmonic Maxwell's equations $(\mathcal{P})$ (see \citeAY{Cessenat:1996:MME};
\citeAY{Kirsch:2015:MTT}; \citeAY{Monk:2003:FEM}; \citeAY{Nedelec:2001:AEE}).

The first step is to introduce a lifting of the boundary data%
\index{lifting of boundary data}
$\Bf$. As $\Bf\in H^{-\frac{1}{2}}(\Divop,\partial\Omega)$, there exists (see Theorem 5.24 of \citeAY{Kirsch:2015:MTT}) a continuous lifting operator $R: H^{-\frac{1}{2}}(\Divop,\partial\Omega)\to H(\Curlop,\Omega)$ such that
\begin{equation}\label{eq.oplift}
R(\Bf)=\BE_0,
\end{equation}
that is a field $\BE_0\in H(\Curlop, \Omega)$ which depends continuously on $\Bf$ such that $\BE_0 \times \Bn=\Bf$ on $\partial \Omega$.
Thus, the field $\tilde{\BE}=\BE-\BE_0$ satisfies the following problem with homogeneous boundary condition:
\begin{equation*}\label{FTM1}
(\tilde{\mathcal{P}})\left\{ \begin{array}{ll}
\Curlop \tilde{\BE}-i  \omega \mu \,  c^{-1}\BH= -\Curlop \BE_0& \mbox{in }  \Omega, \\[10pt]
\Curlop \BH +i  \omega \ep \, c^{-1} \tilde{\BE}= -i  \omega \ep \,  c^{-1}  \BE_0& \mbox{in }  \Omega, \\[10pt]
\tilde{\BE}\times \Bn=0 &  \mbox{on }  \partial \Omega.
\end{array} \right.
\end{equation*}

Now multiplying the second Maxwell's equation%
\index{Maxwell's equations}
of $(\tilde{\mathcal{P}})$ by a test function $\BGy \in H_{0}(\Curlop,\Omega)$, integrating by parts and then eliminating the unknown $\BH$ by using  the first Maxwell's equation, we get the following variational formula%
\index{variational formula}%
\index{Maxwell's equations!variational formula}
for the electrical field $\tilde{\BE}$:
\begin{equation}\label{Feq.var}
\int_{\Omega} - c^2 \, (\mu \omega)^{-1} \, \Curlop  \tilde{\BE} \cdot  \overline{\Curlop  \BGy} + \omega \ep \,  \tilde{\BE} \cdot \overline{ \BGy }\, \dint \Bx= \int_{\Omega} c^2 \, (\mu \omega)^{-1}\Curlop  \BE_0 \cdot \overline{ \Curlop  \BGy} -\omega \ep  \,   \BE_0 \cdot \overline{ \BGy} \,\dint \Bx,
\end{equation}
satisfied by all $\BGy \in H_{0}(\Curlop,\Omega)$.
The variational formula (\ref{Feq.var}) and the problem $(\mathcal{P})$ are equivalent.
\begin{Pro}\label{FPro.eq}
$\tilde{\BE}\in H_{0}(\Curlop,\Omega)$ is a solution of the variational formulation (\ref{Feq.var}) if and only if $\big(\BE=\tilde{\BE} + \BE_0, \BH=c \,(i\mu \omega)^{-1}\Curlop (\tilde{\BE} + \BE_0 ) \big )\in H(\Curlop,\Omega)^2$ satisfy the problem $(\mathcal{P})$.
\end{Pro}
\begin{proof}
This proof is standard. For more details, we refer to the demonstration of the lemma 4.29 in \citeAPY{Kirsch:2015:MTT}.
\end{proof}
For all $\BZ \in (\bbC^+)^4$, we introduce  the sesquilinear form:
$$
a_{\BZ}(\BGf, \BGy)=\int_{\Omega} - c^2(\mu \omega)^{-1} \, \Curlop  \BGf \cdot \overline{\Curlop  \BGy} + \omega \ep \,  \BGf \cdot \overline{\BGy} \, \dint \Bx,
$$
defined on $H_{0}(\Curlop,\Omega)^2$.
One easily proves by using the Cauchy--Schwarz inequality%
\index{Cauchy--Schwarz inequality}
that:
\begin{equation}\label{Feq.bilin}
|a_{\BZ}(\BGf, \BGy)|\leq \max\big( c^2\left\| (\omega\mu)^{-1}\right\|_{\infty},\left\| \omega \ep\right\|_{\infty}\big)\, \left\|\BGf \right\|_{H(\Curlop,\Omega)} \,\left\|\BGy \right\|_{H(\Curlop,\Omega)},
\end{equation}
where $\left\|\cdot\right\|_{\infty}$ denotes the $L^{\infty}$ norm.
Thus, $a_{\BZ}$ is continuous and as such it allows us to define a continuous linear operator $A_{\BZ}\in L(H_0(\Curlop,\Omega), H_0(\Curlop,\Omega)^{*})$ by
\begin{equation}\label{Feq.defopAZ} \left\langle  A_{\BZ} \BGf,\BGy\right\rangle_{ H_0(\Curlop,\Omega)}=a_{\BZ}(\BGf, \overline{\BGy}), \ \forall  \BGf, \BGy \in H_0(\Curlop,\Omega),
\end{equation}
where $\left\langle \cdot , \cdot\right\rangle_{ H_0(\Curlop,\Omega)}$ stands for the duality product between $H_0(\Curlop,\Omega) $ and its dual $H_0(\Curlop,\Omega) ^{*}$.
We now introduce the antilinear form $l_{\BZ}(\BE_0)(\cdot)$:
$$
l_{\BZ}(\BE_0)(\BGy)=\int_{\Omega} c^2 \, (\mu \omega)^{-1}\Curlop  \BE_0 \cdot \overline{ \Curlop  \BGy} -\omega \ep \, \BE_0 \cdot  \overline{ \BGy} \,\dint \Bx, \, \forall \BGy \in H_0(\Curlop,\Omega).
$$
In the same way as (\ref{Feq.bilin}), one can easily check:
$$
|l_{\BZ}(\BE_0)(\BGy) | \leq \max\big(c^2\,\left\| (\omega\mu)^{-1}\right\|_{\infty},\left\| \omega \ep\right\|_{\infty}\big)\, \left\| \BE_0  \right\|_{H(\Curlop,\Omega)} \,\left\|  \BGy \right\|_{H(\Curlop,\Omega)}.
$$
Hence, the linear operator $L_{\BZ}:H(\Curlop,\Omega)\to H_0(\Curlop,\Omega)^{*}$ defined by
\begin{equation}\label{Feq.lin}
\left\langle  L_{\BZ} \BE_0,\BGy\right\rangle_{ H_0(\Curlop,\Omega)}=l_{\BZ}(\BE_0)( \overline{\BGy}), \ \forall  \BE_0 \in H(\Curlop,\Omega)\mbox{ and }\forall \, \BGy \in H_0(\Curlop,\Omega),
\end{equation}
is well-defined and continuous.
Thus, we deduce from the relations (\ref{Feq.bilin}) and (\ref{Feq.lin}) that the variational formula  (\ref{Feq.var}) is equivalent to solve the following infinite dimensional linear system
\begin{equation}\label{Feq.syst}
A_{\BZ}\, \tilde{\BE} = L_{\BZ}\, \BE_0.
\end{equation}
\begin{Pro}\label{Fprop.invAZ}
If $\BZ\in (\bbC^{+})^{4}$, then the operator $A_{\BZ}$ is an isomorphism from $H_0(\Curlop,\Omega)$ to $H_0(\Curlop,\Omega)^{*}$ and its inverse $A_{\BZ}^{-1}$ depends analytically on $\BZ$ in $(\bbC^{+})^{4}$.
\end{Pro}
\begin{proof}
Let $\BZ$ be in $(\bbC^{+})^{4}$. The invertibility of $A_{\BZ}$ is an immediate consequence of the Lax--Milgram Theorem.%
\index{Lax--Milgram theorem}
Indeed, the coercivity%
\index{coercivity property}
of the sesquilinear form $a_{\BZ}$ derives from the passivity hypothesis (\ref{Feq.pass}) of the material:
\begin{equation*}
|a_{\BZ}(\BGf,\BGf)| \geq  \operatorname{Im} \big( a_{\BZ}(\BGf,\BGf) \big) \geq   \alpha  \left\|\BGf \right\|_{H(\Curlop,\Omega)}^2,  \  \forall \BGf \in H_0(\Curlop,\Omega),
\end{equation*}
where $\alpha=\min\big(\operatorname{Im}( \omega \ep_1) , \operatorname{Im}( \omega \ep_2), - c^2 \operatorname{Im}( \,(\omega \mu_1 )^{-1}),- c^2 \operatorname{Im}(\omega \mu_2 )^{-1})\big)>0$.

Now the analyticity in $\BZ$ of the operator $A_{\BZ}^{-1}$ is proved as follows. First, one can verify easily that for all $\BGf,\Bpsi\in H_0(\Curlop,\Omega)$, the sesquilinear form  $a_{\BZ}(\BGf, \overline{ \Bpsi})$ depends analytically on each component of $\BZ$ when the others are fixed. It follows immediately from this and Theorem \ref{Th.opanalytic} that the operator $A_{\BZ}$ [defined by the relation (\ref{Feq.defopAZ})] is analytic in the operator norm of $L(H_0(\Curlop,\Omega),H_0(\Curlop,\Omega)^{*})$ and hence by Theorem \ref{FThm.Hartog} (Hartogs' Theorem)%
\index{Hartogs' theorem}
it is analytic in $\BZ$ in the open set $(\bbC^{+})^{4}$. Thus, since $A_{\BZ}$ is an isomorphism which depends analytically on $\BZ$ in the open set $(\bbC^{+})^{4}$, then by Theorem \ref{FTh.analyinvop} its inverse $A_{\BZ}^{-1}$ depends analytically on $\BZ$ in $(\bbC^{+})^{4}$.
\end{proof}
\noindent Using Theorem \ref{FThm.Hartog} again, one can easily check in the same way as for the operator $A_{\BZ}$ that the operator $L_{\BZ}$ defined by (\ref{Feq.lin}) is also analytic in $\BZ$ in $(\bbC^{+})^{4}$.
Hence, the variational formula $(\ref{Feq.var})$ admits a unique solution:
\begin{equation}\label{Feq.deftildeE}
\tilde{\BE}=A_{\BZ}^{-1}  L_{\BZ} \,\BE_0=A_{\BZ}^{-1}  L_{\BZ}\, R(\Bf)
\end{equation}
which depends continuously on the data $\Bf$ and analytically on $\BZ$ in $(\bbC^{+})^{4}$.
\begin{Cor}\label{FCor.TZ}
The linear operator $T_{\BZ}:H^{-\frac{1}{2}}(\Divop,\partial\Omega) \to H(\Curlop,\Omega)^2$ which maps the data $\Bf$ to the solution  $(\BE, \BH)\in H(\Curlop,\Omega)^2$ of $(\mathcal{P})$ is well-defined, continuous and depends analytically on $\BZ$ in $(\bbC^{+})^{4}$.
\end{Cor}
\begin{proof}
This result is just a consequence of Propositions \ref{FPro.eq} and \ref{Fprop.invAZ}
which prove that the time-harmonic Maxwell's equations $(\mathcal{P})$ admits a unique solution $(\BE, \BH)=T_{\BZ}(\Bf) \in H(\Curlop,\Omega)^2$ for data $\Bf \in H^{-\frac{1}{2}}(\Divop,\partial\Omega)$ where the linear operator $T_{\BZ}$ is defined by the following relation:
\begin{equation}\label{Feq.defTZ}
T_{\BZ}(\Bf)=\big(R(\Bf)+\tilde{\BE},c\, (i\mu \omega)^{-1}\Curlop (\tilde{\BE} + R(\Bf) \big), \ \forall \Bf \in H^{-\frac{1}{2}}(\Divop,\partial\Omega),
\end{equation}
where $\tilde{\BE}=A_{\BZ}^{-1}  L_{\BZ} \, R (\Bf)$ by the relation (\ref{Feq.deftildeE}) and $R$ stands for the lifting operator defined in \eqref{eq.oplift}. With the relation (\ref{Feq.defTZ}), the continuity of $T_{\BZ}$ with respect to $\Bf$ and its analyticity with respect to $\BZ$ follow immediately from the corresponding properties of the operator $A_{\BZ}^{-1}$, $L_{\BZ}$ and $R$.
\end{proof}

We now introduce the tangential component trace operator $\gamma_T: H(\Curlop,\Omega) \to  H^{-\frac{1}{2}}(\Curlop,\partial\Omega) $ defined by:
\begin{equation}\label{Fdef.optrace}
\gamma_T(\BH)=\Bn \times ( \BH \times \Bn)_{\partial \Omega}, \ \forall \BH\in H(\Curlop,\Omega),
\end{equation}
which is continuous (see Theorem 5.24 of \citeAY{Kirsch:2015:MTT}) and the continuous linear operator $P:H(\Curlop,\Omega)^2\to H(\Curlop,\Omega) $ defined by:
\begin{equation}\label{Fdef.opp}
P(\BE,\BH)=\BH, \ \forall \BE, \BH\in H(\Curlop,\Omega).
\end{equation}

This gives us the following operator representation of the electromagnetic DtN map defined in (\ref{Feq.DtN}).
\begin{Pro}(Electromagnetic Dirichlet-to-Neumann map)%
\index{Dirichlet-to-Neumann map!electromagnetism}
The electromagnetic DtN map $\Lambda_{\BZ}:H^{-\frac{1}{2}}(\Divop,\partial\Omega) \to  H^{-\frac{1}{2}}(\Curlop,\partial\Omega)$ is the continuous linear operator defined by the composition of the continuous linear operators $\gamma_T$ in (\ref{Fdef.optrace}), $P$ in (\ref{Fdef.opp}) and $T_{\BZ}$ in (\ref{Feq.defTZ}) by
\begin{equation}\label{Feq.DtNbis}
\Lambda_{\BZ}(\Bf)= i \, \gamma_T \,P \,T_{\BZ}(\Bf), \ \forall \Bf \in H^{-\frac{1}{2}}(\Divop,\partial\Omega).
\end{equation}
\end{Pro}
\begin{proof}
Let $\Bf\in H^{-\frac{1}{2}}(\Divop,\partial\Omega)$. Then $(\BE,\BH)=T_{\BZ}(\Bf)$ is the solution of problem ($\mathcal{P}$). Thus, by definition of $\gamma_T$ and $P$ we have $PT_{\BZ}(\Bf)=\BH$ and hence $i\gamma_TPT_{\BZ}(\Bf)=i\Bn\times(\BH\times\Bn)_{\partial\Omega}$. Therefore, by the definition (\ref{Feq.DtN}) of the DtN map we have $\Lambda_{\BZ}(\Bf)=i\gamma_TPT_{\BZ}(\Bf)$. The fact that $\Lambda_{\BZ}$ is a continuous linear operator follows immediately from this representation. This completes the proof.
\end{proof}

We can now derive the regularity properties of the DtN map $\Lambda_{\BZ}$  by expressing this operator in terms of the
operator $T_{\BZ}$.
\index{Dirichlet-to-Neumann map!well-posedness and regularity}
The analyticity of the DtN map $\Lambda_{\BZ}$ with respect to $\BZ$ in $(\bbC^{+})^4$ is now an immediate consequence of the fact that $\Lambda_{\BZ}$ is the composition of two continuous linear operators $ i\, \gamma_T$ and $P$ independent of $\BZ$ with the continuous linear operator $T_{\BZ}$ which is analytic in $\BZ$ (see Corollary \ref{FCor.TZ}).

Finally, to prove the positivity of $\operatorname{Im}\left\langle\Lambda_{\BZ} \Bf, \overline{\Bf} \right\rangle $, we apply Green's identity%
\index{Green's identity}
(\ref{Feq.duality}) to the solution $(\BE, \BH)$ of the problem $(\mathcal{P})$ for any nonzero data $\Bf \in H^{-\frac{1}{2}}(\Divop,\partial\Omega)$. It yields
$$
i \int_{\Omega} \overline{\BE}\cdot \Curlop \BH - \BH \cdot \overline{\Curlop \BE }\, \dint \Omega=i \left\langle \Bn\times (\BH\times \Bn) , \overline{\BE }\times \Bn \right\rangle=\left\langle\Lambda_{\BZ}\, \Bf,\overline{ \Bf} \right\rangle.
$$
Since $(\BE, \BH)$ is a solution of the time-harmonic Maxwell equations $(\mathcal{P})$, we can rewrite this last relation as:
\begin{equation}\label{Feq.varDtN}
\int_{\Omega} \omega \ep \, c^{-1} \left|\BE\right|^2 -  \overline{ \omega \mu } \, c^{-1} \left|\BH \right|^2 \dint \Bx =\left\langle\Lambda_{\BZ}\, \Bf,\overline{ \Bf} \right\rangle.
\end{equation}
By taking the imaginary part of (\ref{Feq.varDtN}) and using the passivity hypothesis (\ref{Feq.pass}) of the materials which compose the medium $\Omega$, we get the positivity of $\operatorname{Im}\left\langle\Lambda_{\BZ} \Bf,  \overline{\Bf}  \right\rangle $ (\ref{Feq.pos}) (since $(\BE, \BH) \neq (0,0)$ for $\Bf\neq 0$) and it follows immediately that the function $h_{\Bf}$ defined by (\ref{Feq.herg}) is a Herglotz function of $\BZ$. 

\subsection{Extensions of Theorem \ref{FTh.DtN} to anisotropic and continuous media}\labsect{subsec.ext}
Here we first extend Theorem \ref{FTh.DtN} to the case of a medium $\Omega$ composed by $N$ anisotropic%
\index{anisotropic media}
homogeneous phases. Therefore, the dielectric permittivity $\BGve(\Bx)$ and magnetic permeability $\BGm(\Bx)$ 
are now  $3\times 3$ tensor-valued functions 
of $\Bx$, which take for $j=1,\cdots,N$ the constant values $\BGve_j$ and $\BGm_j$ in the $j$th material. Again, each material is supposed to be passive, in the sense that $\operatorname{Im}(\omega  \BGve_j)$ and $\operatorname{Im}(\omega  \BGm_j)$ have to be positive tensors for all $j=1,\cdots,N$ (see \citeAY{Milton:2002:TOC}, \citeAY{Welters:2014:SLL}, \citeAY{Bernland:2011:SRC}, \citeAY{Gustafsson:2010:SRP}).

First, we  want to emphasize that besides the fact that $\BGve$ and $\BGm$ are now tensor-valued, the time-harmonic Maxwell's equations $(\mathcal{P})$ in $\Omega$ and its associated functional spaces remain the same.
Moreover, as the vector space $M_3(\bbC)$ is isomorphic to $\bbC^{9}$, we prove exactly in the same way that
the DtN map $\Lambda_{\BZ}$ defined by (\ref{Feq.DtN}) is  well-defined, is linear continuous with respect to $\Bf$, and is an 
analytic function with respect to $\BZ$, where $\BZ$ is here the vector of the $18N$ coefficients which are the elements (in some basis) of 
the tensors $\omega \BGve_j$ and $\omega  \BGm_j$ for $j=1,\cdots,N$, in the open set $\mathcal{O}$ of $\bbC^{18N}$ characterized by the passivity relation (\ref{Feq.pass}).
As $\mathcal{O}$ is isomorphic to the open set $(M_3^+(\mathbb{C}))^{2N}$, this is equivalent to say (as it is explained in the last paragraph of the subsection \ref{subsection2.1}) that $\Lambda_{\BZ}$ is an analytic function of the vector $\BZ$, whose components are now those of the permittivity tensors $\omega \BGve_j$ and  permeability $\omega \BGm_j$ (for $j=1,\cdots,N$) in each phase.
Using the passivity assumption%
\index{passivity}
which is associated with the elements of $(M_3^+(\mathbb{C}))^{2N}$, one proves also identically the relation (\ref{Feq.pos}) on the DtN map for all $\BZ\in (M_3^+(\mathbb{C}))^{2N}$. 

The problem is now to define the notion of a Herglotz function.%
\index{Herglotz functions}
Indeed, when the tensors $\BGve_j$ and $\BGm_j$ of each composite are not all diagonal, it is not
possible anymore to define the DtN map
as a multivariate Herglotz function $h_{\Bf}$ on some copy of the upper-half plane: $(\mathbb{C}^+)^n$. The major obstruction to this construction is based on the simple observation that off-diagonal elements of a matrix in $M_3^+(\mathbb{C})$ will not necessarily have a positive imaginary part.
Nevertheless, it is natural to define a Herglotz function which maps points $\BZ$ represented by a $2N$-tuple of matrices $\BL_1',\BL_2', \ldots \BL_{2N}'$ with positive definite imaginary parts, i.e.,
$\BZ=(\BL_1',\BL_2', \ldots \BL_{2N}')$, to the upper half-plane.

One way to preserve the Herglotz property is to use a trajectory method%
\index{trajectory method}
(see \citeAY{Bergman:1978:DCC} and
Section 18.6 of \citeAY{Milton:2002:TOC}), in other words, consider an analytic function $s\mapsto \BZ(s)$ from $\bbC^{+}$ into $(M_3^+(\mathbb{C}))^{2N}$, i.e., a trajectory in one complex dimension (a surface in two real directions). Then, along this trajectory, we obtain immediately that the function
\begin{equation}\label{eq.Hergfreq} h_{\Bf}(s)=   \,\left\langle  \Lambda_{\BZ(s)} \Bf,\overline{ \Bf} \right\rangle, \ \forall \Bf \in H^{-\frac{1}{2}}(\Divop,\partial\Omega),
\end{equation}
is a Herglotz function (see Definition \ref{FDef.Herg1}) of $s$ in $\bbC^{+}$: analyticity follows from the fact that analyticity
is preserved under composition of analytic functions, while, when $s$ has positive imaginary part, nonnegativity
of the imaginary part of $h_{\Bf}(s)$ follows from the fact that
$\BZ(s)$ lies in the domain where the imaginary part of the operator $\Lambda_{\BZ(s)}$ is positive semi-definite.

 A particularly interesting trajectory for electromagnetism, in an $N$-phase material, is the trajectory
\begin{equation*} s=\Go\to \BZ(\omega)=(\Go\BGve_1(\Go),\Go\BGve_2(\Go),\ldots,\Go\BGve_N(\Go),\Go\BGm_1(\Go),\Go\BGm_2(\Go),\ldots,\Go\BGm_N(\Go)),
\end{equation*}
where $\BGve_j(\Go)$ and $\BGm_j(\Go)$ are the physical electric permittivity tensor and physical magnetic permeability tensor of the actual material constituting phase $i$
as functions of the frequency $\Go$.
Due to the passive nature of these materials the trajectory maps $\Go$ in the upper half plane $\bbC^{+}$ into a trajectory in $(M_3^+(\mathbb{C}))^{2N}$ .
The physical interest about this
trajectory is that one can in principle measure $\Lambda_{\BZ(\omega)}$ along it, at least for real frequencies $\Go$.

In the case of the trajectory method, one can easily generalize Theorem \ref{FTh.DtN}  to continuous anisotropic composites where $\BGve$ and $\BGm$ are matrix-valued functions of both variables $(\Bx,\omega)\in \Omega \times \bbC^{+} $. In this case, we suppose that

\begin{itemize}
\item (H1) For all $\omega \in \bbC^{+}$,  $\BGve(\cdot,\omega)$ and  $\BGm(\cdot,\omega) $ are $L^{\infty}$ matrix-valued functions on $\Omega$ which are locally bounded in the variable $\omega$, in other words, we suppose that there exists $\delta>0$ such that the open ball of center $\omega$ and radius $\delta$: $B(\omega,\delta)$ is included in $\bbC^{+}$ and that
\begin{equation}\label{eq.dom}
\sup_{z \in B(\omega,\delta)} \Vert\BGve(\cdot,z)\Vert_{\infty}<\infty  \mbox{ and }  \sup_{z \in B(\omega,\delta)} \Vert\BGm(\cdot,z)\Vert_{\infty}<\infty
\end{equation}
\item (H2) The composite is passive which implies that for almost every $x\in \Omega$, the functions $\omega \mapsto \omega\BGve(\Bx,\omega)$ and  $\omega\mapsto \omega\BGm(\Bx,\omega)$ are analytic functions from  $\bbC^+ $ to $M_3^+(\mathbb{C})$  (see  Section 11.1 of \citeAY{Milton:2002:TOC}).
\item (H3) For all $ \omega \in \bbC^{+}$, there exists $C_{\omega}>0$ such that
\begin{equation}\label{eq.conscoerc}
\operatorname{ess}\inf \limits_{\Bx\in \Omega}\operatorname{Im}(\omega \BGve(\cdot,\omega)) \geq C_{\omega} \operatorname{Id} \  \mbox{ and } \ \operatorname{ess}\inf \limits_{\Bx \in \Omega}-(\operatorname{Im}(\omega \BGm(\cdot,\omega))^{-1} \geq C_{\omega} \operatorname{Id}.
\end{equation}
\end{itemize}
\begin{Rem}
  The hypotheses (H1) and (H3) may seem complicated but they are satisfied for instance when $\BGve$ and $\BGm$ are continuous functions of both variables $(\Bx,\omega) \in \operatorname{cl}\Omega\times  \mathbb{C}^+$, where $\operatorname{cl}\Omega$ denotes  the closure of $\Omega$. In that case, one can
  see immediately that $(H1)$ is satisfied.
Moreover, if we assume also that the passivity assumption (H2) holds on $\operatorname{cl}\Omega$ (instead of just $\Omega$), the hypothesis (H3) is also satisfied since
the functions $\operatorname{Im}(\omega \BGve(\cdot,\omega))$ and $-(\operatorname{Im}(\omega \BGm(\cdot,\omega)))^{-1}$ are continuous functions on a compact set and thus reach their minimum value which is a positive matrix.
\end{Rem}

Under these hypotheses, Theorem \ref{FTh.DtN} remains valid: the function $h_{\Bf}(\omega)$ given by (\ref{Feq.herg}) is a Herglotz function of the frequency (by interpreting each formula of Section  \ref{Sec.boundedmedia} with $\BZ=\omega\in  \mathbb{C}^+$ as a new analytic variable and $\BGve$ and $\BGm$ as matrix valued functions of the variables $(\Bx,\omega)$). Moreover, the proof is basically the same as the one in Subsection \ref{sec.proof}. We just make precise here
the justification of some technical points which appear when one reproduces this proof in this framework.

We first remark that the assumption (H1) on the tensors $\BGve$ and $\BGm$ implies that the tensors $\omega \BGve(\cdot,\omega)$ and $(\omega \BGm(\cdot,\omega))^{-1}$ are bounded functions of the space variable $\Bx$. Thus the bilinear form $a_{\omega}$ remains continuous and the operators $A_{\omega}$ and $L_{\omega}$ are still well-defined and continuous.
With the coercivity%
\index{coercivity property}
hypothesis (H3), one can easily check that $ \forall \BGf \in H_0(\Curlop,\Omega),$
$$
|a_{\omega}(\BGf,\BGf)|  \geq  C_{\omega}  \left\|\BGf \right\|_{H(\Curlop,\Omega)}^2,
$$
and apply again the Lax--Milgram theorem%
\index{Lax--Milgram theorem}
to show the invertibility of $A_{\omega}$.
Then, the analyticity of the operators $A_{\omega}$ and $L_{\omega}$ with respect to $\omega$ in $\bbC^{+}$ is still obtained (thanks to the relations  (\ref{Feq.defopAZ}) and (\ref{Feq.lin})) from their weak analyticity (see Theorem \ref{Th.opanalytic}).
This weak analyticity is proved by using Theorem \ref{Th.derivsum} to show the analyticity of the integrals which appear in the expression of  $\left\langle  A_{\omega} \BGf, \BGy \right\rangle_{ H_0(\Curlop,\Omega)}$ and  $\left\langle  L_{\omega} \BE_0,\BGy \right\rangle_{ H_0(\Curlop,\Omega)}$ for $\Bphi$, $\Bpsi\in H_0(\Curlop,\Omega) $ and $\BE_0\in H(\Curlop,\Omega) $ (since the assumptions (H1) and (H2) imply the hypotheses of Theorem  \ref{Th.derivsum}).
Then the analyticity of $A_{\omega}^{-1}$ is proved by using again Theorem  \ref{FTh.analyinvop} and the rest of the proof follows by the same arguments as in the isotropic case.

\section{Herglotz functions associated with anisotropic  media}\label{sec.anisopHerg}

A theory of Herglotz functions%
\index{Herglotz functions}
directly defined on tensors and not only on scalar variables is particularly useful
in the domain of bodies containing anisotropic materials such as, for instance, sea ice%
\index{sea ice}
(see \citeAY{Carsey:1992:MRS,Stogryn:1987:GTD,Golden:1995:BCP,Golden:2009:CCM,Gully:2015:BCP}) or in electromagnetism where it will even extend to complicated media such as gyrotropic materials%
\index{gyrotropic materials}
for which the dielectric tensors and magnetic tensors are anisotropic but not symmetric (as there is no reciprocity principle in such media, see \citeAY{Landau:1984:ECM}). The idea for Herglotz representations of the effective moduli of anisotropic materials  was first put forward in the appendix E of \citeAPY{Milton:1981:BTO}, and was studied in depth in Chapter 18 of  \citeAPY{Milton:2002:TOC}, see also \citeAPY{Barabash:1999:SRE}. In connection with sea ice,%
\index{bounds!sea ice}
one is particularly interested in bounds where the moduli
are complex: such bounds are an immediate corollary of appendix E of \citeAPY{Milton:1981:BTO} and series expansions of the effective
conductivity (\citeAY{Willemse:1979:ECP}; \citeAY{Avellaneda:1990:ECA}) that are contingent on assumptions about the polycrystalline geometry,
and more generally are obtained (even for viscoelasticity with anisotropic phases) in \citeAPY{Milton:1987:MCEb} (to make the connection,
see the discussion in Section 15 of the companion paper \cite{Milton:1987:MCEa}), and also see the bounds (16.45) in \citeAPY{Milton:1990:CSP}.
Explicit calculations were made by \citeAPY{Gully:2015:BCP} and are in good agreement with sea ice measurements.

The trajectory method%
\index{trajectory method}
provides the desired representation, as shown in Section 18.8 of \citeAPY{Milton:2002:TOC}. We slightly
modify that argument here. Given any tensors $\BL_j=\Go\BGve_j$ and $\BL_{j+N}=\Go\BGm_j$, for $j=1,2,\ldots N$, which we assume to have strictly positive definite imaginary parts, and given a reference tensor $\BL_0$ which is real, but not necessarily positive definite, define the real matrices
\beq \BA_j  =  \Real[(\BL_0-\BL_j)^{-1}], \quad \BB_j=\Imag[(\BL_0-\BL_j)^{-1}],\quad j=1,2,\ldots,2N,
\eeq{F.001}
where, according to our assumption, $\BB_j$ is positive definite for each $j$. Then consider the trajectory
\beq \BZ(s)=(\BL_1'(s),\BL_2'(s), \ldots \BL_{2N}'(s)), \quad {\rm where}\,\,\BL_j'(s)= \BL_0-(\BA_j+s\BB_j)^{-1}. \eeq{F.002}
Each of the matrices $\BL_j'(s)$ have positive definite imaginary parts when $s$ is in  $\mathbb{C}^+$, and so $ \BZ(s)$ maps $\mathbb{C}^+$
to $(M_3^+(\mathbb{C}))^{2N}$. Furthermore, by construction our trajectory passes through the desired point at $s=i$:
\beq \BZ(i)=(\Go\BGve_1, \Go\BGve_2, \ldots, \Go\BGve_N,  \Go\BGm_1, \Go\BGm_2, \ldots \Go\BGm_N).
\eeq{F.003}
Now $\Lambda_{\BZ(s)}$ is an operator valued Herglotz function of $s$, and so has an integral representation involving a positive semi-definite operator-valued measure deriving
from the values  that $\BZ(s)$ takes when $s$ is just above the positive real axis.  That measure in turn is linearly dependent on the measure
derived from the values that
 $\Lambda_{\BZ(\BL_1',\BL_2', \ldots \BL_{2N}')}$ takes as imaginary parts of $\BL_j'$ become vanishingly small. Thus $\BZ(i)$ can be expressed directly in terms of this latter measure, involving an integral kernel $\BK_{\BL_0}(\BL_1,\BL_2, \ldots \BL_{2N},\BL_1',\BL_2', \ldots \BL_{2N}')$ that is singular with support that is concentrated on the trajectory which is traced by
$(\BL_1'(s),\BL_2'(s), \ldots \BL_{2N}'(s))$ as $s$ is varied along the real axis.  The
formula for $\Lambda_{\BZ(i)}$ obtained from the above prescription can be rewritten (informally) as
\beq \Lambda_{\BZ(\BL_1,\BL_2, \ldots \BL_{2N})}
=\int \BK_{\BL_0}(\BL_1,\BL_2, \ldots \BL_{2N},\BL_1',\BL_2', \ldots \BL_{2N}')\,d\Bm(\BL_1',\BL_2', \ldots \BL_{2N}').
\eeq{F.003a}
An explicit formula could be obtained for the kernel
$\BK_{\BL_0}(\BL_1,\BL_2, \ldots \BL_{2N},\BL_1',\BL_2', \ldots \BL_{2N}')$, which is non-zero except on the
path traced out by $(\BL_1'(s),\BL_2'(s), \ldots \BL_{2N}'(s))$ as $s$ varies over the reals,
and this path depends on $\BL_0$ and the moduli of
$\BL_1,\BL_2, \ldots \BL_{2N}$.
The measure $d\Bm(\BL_1',\BL_2', \ldots \BL_{2N}')$ is derived from the values that
 $\Lambda_{\BZ(\BL_1',\BL_2', \ldots \BL_{2N}')}$ takes as imaginary parts of $\BL_j'$ become vanishingly small.

The trajectory
method has been unjustly criticised for failing to separate the dependence of the function (in this case
$\Lambda_{\BZ(\BL_1,\BL_2, \ldots \BL_{2N})}$) on the moduli (in this case the tensors $\BL_1,\BL_2, \ldots \BL_{2N}$)
from the dependence on the geometry, which is contained in the measure (in this case
derived from the values that $\Lambda_{\BZ(\BL_1',\BL_2', \ldots \BL_{2N}')}$ takes as imaginary parts of $\BL_j'$ become
vanishingly small.) But we see that \eq{F.003a} makes such a separation.

Now there are differences between this representation and standard
representation formulas for multivariate Herglotz functions, but the main difference is that the kernel
$\BK_{\BL_0}(\BL_1,\BL_2, \ldots \BL_{2N},\BL_1',\BL_2', \ldots \BL_{2N}')$, unlike the Szeg{\H{o}} kernel%
\index{Szeg{\H{o}} kernel}
which enters the
polydisk representation%
\index{polydisk representation}
of \citeAPY{Koranyi:1963:HFP}, is singular, being concentrated on this trajectory. However, for each choice of $\BL_0$ there is a representation, each involving a kernel with support on a different trajectory
and so one can average the representations over the matrices $\BL_0$ with
any desired smooth nonnegative weighting, to obtain a family of representations with less singular kernels that are the average
over $\BL_0$ of $\BK_{\BL_0}(\BL_1,\BL_2, \ldots \BL_{2N},\BL_1',\BL_2', \ldots \BL_{2N}')$. This nonuniqueness in the choice of kernel reflects
the fact that the measure satisfies certain Fourier constraints on the polydisk (see \citeAPY{Rudin:1969:FTP}).

Another approach to considering of the notion of Herglotz functions in anisotropic multicomponent media
is as follows. It will consist of proving that $(M_3^+(\mathbb{C}))^{2N}$ is isometrically isomorphic to a tubular domain
(defined below) of $\mathbb{C}^{18N}$ and use it to extend the definition of Herglotz functions via the theory of holomorphic functions on tubular domains with nonnegative imaginary part from \citeAY{Vladimirov:2002:MTG}.

As we have already mentioned in the introduction to this chapter and at the end of Subsection \sect{subsec.ext}, this extended definition is significant because these multivariate functions provide a deep connection to the theory of multivariate passive linear systems%
\index{multivariate passive linear systems}
as described in Section 20 of \citeAY{Vladimirov:2002:MTG}, for instance, and in the study of anisotropic composites (e.g., sea ice or gyrotropic media).  In addition to this, such an extension may allow for a more general approach of the efforts of \citeAPY{Golden:1985:BEP} and \citeAPY{Milton:1990:RCF} to derive integral representations of multivariate Herglotz functions in $(\bbC^{+})^N$, beyond that provided in Section 18.8 of  \citeAPY{Milton:2002:TOC}, or for deriving bounds in the theory of composites using the analytic continuation method (see Chapter 27 in \citeAY{Milton:2002:TOC}).

Let us first introduce the definition of a tubular domain%
\index{tubular domain}
from Chapter 2, Section 9 of \citeNP{Vladimirov:2002:MTG}.
\begin{Def}\label{eq.defTub}
Let $\Gamma$ be a closed convex acute cone in $\bbR^N$ with vertex at ${\bf 0}$. We denote by $C=\operatorname{int}(\Gamma^{*})$, where $\Gamma^{*}$ stands for the dual of $C$ (in the sense of cones' duality) and $\operatorname{int}(\Gamma^{*})$ denotes the (topological) interior of $\Gamma^{*}$. Thus, $C$ is  an open, convex, nonempty cone. Then, a tubular domain in $\mathbb{C}^N$ with base $C$ is defined as:
$$
\mathcal{T}=\bbR^N+iC=\left\{ \Bz=\Bx+i\By | \Bx \in \bbR^N  \mbox{ and } \By \in C\right\}.
$$
\end{Def}

We will now show that $M_N^{+}(\bbC)$ is isometrically isomorphic to a tubular domain $\mathcal{T}^{C}$ of $\bbC^{N^2}$ [see Proposition \ref{prop.euclidanspace}, Proposition \ref{prop.cone}, and Theorem \ref{ThmTubularDomain} below (which we state without proof as they are easily verified)]. Toward this purpose, we first  use the decomposition
$$
\BM=\frac{\BM+\BM^{*}}{2}+i \, \frac{\BM-\BM^{*}}{2i}, \, \forall \BM\in M_N(\bbC),
$$
to parameterize the space $M_N^{+}(\bbC)$ as
$$
M_N^{+}(\bbC)=\left\{\BM_1+i\, \BM_2| \BM_1\in H_N(\bbC) \mbox{ and } \BM_2\in H_N^{+}(\bbC) \right\},
$$
where $H_N(\bbC)$ and $H_N^{+}(\bbC) $ denote the sets of Hermitian and positive definite Hermitian matrices, respectively.

Then, we recall with the two following propositions that $H_N(\bbC)$ is a Euclidean space and $H_N^{+}(\bbC)$ is a cone in $H_N(\bbC)$ with some remarkable properties that we will use to construct the basis $C$ of our tubular region. First, denote the standard orthonormal basis vectors of $\mathbb{R}^N$ by $\Be_k$, $k=1,\ldots, N$. With respect to this basis, let $\BE_{kl}$, $1\leq k,l\leq N$ denote the matrices in $M_N(\mathbb{C})$ such that as operators on $\mathbb{C}^N$ are equal to \begin{equation}
\BE_{kl}=\Be_k\Be_l^{T} \mbox{ for } k,l\in \{1,\ldots,N\},
\end{equation}
i.e., $\BE_{kl}$ is the $N\times N$ matrix with $1$ in the $k$th row, $l$th column and zeros everywhere else.
\begin{Pro}\label{prop.euclidanspace}
The Hermitian matrices $H_N(\bbC)$ endowed with the inner product:
\begin{equation}
(\BA,\BB)_{H_N(\bbC)}=\operatorname{Tr}(\BA\BB),\quad \forall \BA,\BB \in H_N(\bbC)
\end{equation}
is a Euclidean space of dimension $N^2$. Furthermore, an orthonormal basis of this space is given by the matrices:
\begin{gather}\label{eq.orthbasis}
\BE_{kk} ~\mbox{ for }~ k\in \{1,2,3,...,N\},\\
\frac{1}{\sqrt 2}(\BE_{kl}+\BE_{lk}),\, \frac{i}{\sqrt 2}(\BE_{kl}-\BE_{lk}) ~\mbox{ for } l,k\in \{1,2,3,...,N\} \mbox{ such that } l <k.
\end{gather}
Moreover, if we denote by $I_N=\{1,\cdots, N\}$, then the linear map $\phi :H_N(\bbC) \mapsto \bbR^{N^2}$ given by
\begin{equation}\label{eq.isom}
\phi(\BA)=\big((A_{kk})_{k\in I_N},(\sqrt 2\operatorname{Re}A_{kl})_{(k,l)\in I_N^2 \mid k<l},(\sqrt 2\operatorname{Im}A_{kl})_{(k,l)\in I_N^2\mid k<l }  \big)\in \bbR^{N^2}.
\end{equation}
which represents the coordinates of $\BA$ in the orthonormal basis (\ref{eq.orthbasis})
defines an isometry in the sense that
\begin{equation}
(\BA, \BB)_{H_N(\bbC)}= \phi(\BA)\cdot \phi(\BB), \ \forall \BA,\BB\in H_{N}(\bbC),
\end{equation}
for $\cdot$ the standard dot product of $\bbR^{N^2}$.
\end{Pro}

\begin{Pro}\label{prop.cone}
In the Euclidean space $H_N(\mathbb{C})$, denote the closure of $H_N^+(\mathbb{C})$ by $\operatorname{cl}H_N^+(\mathbb{C})$ and its (topological) interior by $\operatorname{int}\left( \operatorname{cl}H_N^+(\mathbb{C})\right)$. Then
\begin{equation}
\operatorname{cl}H_N^+(\mathbb{C})=\{\BM\in M_N(\mathbb{C})|\operatorname{Im}\BM\geq 0\},
\end{equation}
i.e., the set all positive semidefinite (Hermitian) matrices in $M_N(\mathbb{C})$. Furthermore, it is a closed, convex, acute cone with vertex at ${\bf 0}$ and is self-dual (in sense of cones' duality). Moreover, $\operatorname{int}\left(\operatorname{cl}H_N^+(\mathbb{C})\right)=H_N^+(\mathbb{C})$ and it is an open, convex, nonempty cone (in the sense of the definition in Sec.\ 4.4 of \citeAY{Vladimirov:2002:MTG}).
\end{Pro}
In particular, it follows immediately from  the Propositions \ref{prop.euclidanspace} and  \ref{prop.cone} that:
\begin{Thm}\label{ThmTubularDomain}
The set $M_N^{+}(\bbC)=H_N(\mathbb{C})+iH_N^+(\mathbb{C})$ is isometrically isomorphic to a tubular domain $\mathcal{T}^C=\mathbb{R}^{N^2}+i \,C$ in $\mathbb{C}^{N^2}$ where $\mathbb{R}^{N^2}$ and $C$ are  respectively defined by the relations $\mathbb{R}^{N^2}=\phi(H_N(\mathbb{C}))$ and $C=\phi(H_N^+(\mathbb{C}))$ for $\phi$ the isometry defined in (\ref{eq.isom}).
\end{Thm}
As the Cartesian product of tubular domains is also a tubular domain, we obtain immediately that the space of tensors $(M_3^+(\mathbb{C}))^{2N}$ associated to a medium $\Omega$ composed of $N$ anisotropic passive composites is isometrically isomorphic to the tubular region $(\mathcal{T}^{C})^{2N}=\mathcal{T}^{C^{2N}}$ (where $\mathcal{T}^{C}$ is the tubular domain defined in the Theorem \ref{ThmTubularDomain}).
Hence, identifying $(M_3^+(\mathbb{C}))^{2N}$ with $\mathcal{T}^{C^{2N}}$ via this isometry allows us to define in Theorem \ref{FTh.DtN} the function
\begin{equation*} h_{\Bf}(\BZ)=   \left\langle  \Lambda_{\BZ} \, \Bf, \overline{\Bf} \right\rangle, \ \forall \Bf \in H^{-\frac{1}{2}}(\Divop,\partial\Omega),  \end{equation*}
on $(M_3^+(\mathbb{C}))^{2N}$ as an Herglotz function of $\BZ$ in the sense that it is an holomorphic function on a tubular domain with a nonnegative imaginary part and it justifies Definition \ref{FDef.Herg1} given in the introduction.

\section*{Acknowledgments}
Graeme Milton is grateful to the National Science Foundation for support through the Research Grant DMS-1211359. Aaron Welters is grateful for the support from the U.S.\ Air Force Office of Scientific Research (AFOSR) through the Air Force's Young Investigator Research Program (YIP) under the grant FA9550-15-1-0086.

\bibliographystyle{mod-xchicago}



\bibliography{newref,tcbook}


\ifx \bblindex \undefined \def \bblindex #1{} \fi\ifx \bblindex \undefined \def
  \bblindex #1{} \fi
\begin{thebibliography}{}

\ifx \xbblversion \undefined \input xbbl.sty \fi

\bibitem[\protect\citeauthoryear{Agler, McCarthy, and Young}{Agler
  et~al.}{2012}]{Agler:2012:OMF}
\bblauthor{Agler, J.}, \bblauthor{J.~E. McCarthy}, and \bblauthor{N.~J. Young}
  \bblyear{2012}.
\newblock \bbltitle{Operator monotone functions and {L}\"owner functions of
  several variables}.
\newblock {\em \bbljournal{Annals of Mathematics}\/} \bblvolume{176}\penalty0
  (\bblnumber{3}):\penalty0 \bblpages{1783--1826}.
\showEXTRA{%
}

\bibitem[\protect\citeauthoryear{Albanese and Monk}{Albanese and
  Monk}{2006}]{Albanese:2006:ISP}
\bblauthor{Albanese, R.} and \bblauthor{P.~Monk} \bblyear{2006}.
\newblock \bbltitle{The inverse source problem for {M}axwell's equations}.
\newblock {\em \bbljournal{Inverse problems}\/} \bblvolume{22}\penalty0
  (\bblnumber{3}):\penalty0 \bblpages{1023}.
\showEXTRA{%
\showbibdate{\bblbibdate{Thu Sep 24 2015}}
}

\bibitem[\protect\citeauthoryear{Avellaneda and Bruno}{Avellaneda and
  Bruno}{1990}]{Avellaneda:1990:ECA}
\bblauthor{Avellaneda, M.} and \bblauthor{O.~P. Bruno} \bblyear{1990}.
\newblock \bbltitle{Effective conductivity and average polarizability of random
  polycrystals}.
\newblock {\em \bbljournal{Journal of Mathematical Physics}\/}
  \bblvolume{31}\penalty0 (\bblnumber{8}):\penalty0 \bblpages{2047--2056}.
\showEXTRA{%
\showfjournal{\bblfjournal{Journal of Mathematical Physics}}
\showCODEN{\bblCODEN{JMAPAQ}}
\showISSN{\bblISSN{0022-2488}}
\showMR
       {}%
       {\bblMRnumber{91f:82070}}%
       {\bblMRclass{82D25 (78A99)}}
\showbibdate{\bblbibdate{Fri Nov 5 13:45:09 MST 1999}}
\showoriginalref{Avellaneda M. and Bruno O. 1990 Effective conductivity and
  average polarizability of random polycrystals. {\em J. Math. Phys.} {\bf 31},
  2047--2056.}
}

\bibitem[\protect\citeauthoryear{Barabash and Stroud}{Barabash and
  Stroud}{1999}]{Barabash:1999:SRE}
\bblauthor{Barabash, S.} and \bblauthor{D.~Stroud} \bblyear{1999}.
\newblock \bbltitle{Spectral representation for the effective macroscopic
  response of a polycrystal: application to third-order non-linear
  susceptibility}.
\newblock {\em \bbljournal{Journal of Physics: Condensed Matter}\/}
  \bblvolume{11}\penalty0 (\bblnumber{50}):\penalty0 \bblpages{10323}.
\newblock \bblnote{See also the erratum in arXiv:cond-mat/9910246v2}.
\showEXTRA{%
\showbibdate{\bblbibdate{Sun Nov 1 2015}}
}

\bibitem[\protect\citeauthoryear{Baumg{\"a}rtel}{Baumg{\"a}rtel}{1985}]{Baumg:1985:APT}
\bblauthor{Baumg{\"a}rtel, H.} \bblyear{1985}.
\newblock {\em \bbltitle{Analytic perturbation theory for matrices and
  operators}}.
\newblock \bbladdress{Basel, Switzerland}: \bblpublisher{Birkh{\"a}user
  Verlag}.
\newblock \bblpages{427} pp.
\showEXTRA{%
\showISBN{\bblISBN{978-3764316648}}
\showbibdate{\bblbibdate{Fri Sep 4 2015}}
}

\bibitem[\protect\citeauthoryear{Berg}{Berg}{2008}]{Berg:2008:SPBS}
\bblauthor{Berg, C.} \bblyear{2008}.
\newblock {\em \bbltitle{Stieltjes-Pick-Bernstein-Schoenberg and their
  connection to complete monotonicity}}.
\newblock \bbladdress{Castellon, Spain}: \bblpublisher{Positive Definite
  Functions: From Schoenberg to Space-Time Challenges. Ed. J. Mateu and E.
  Porcu. Dept. of Mathematics, University Jaume I}.
\newblock \bblpages{15-45} pp.
\showEXTRA{%
}

\bibitem[\protect\citeauthoryear{Bergman}{Bergman}{1978}]{Bergman:1978:DCC}
\bblauthor{Bergman, D.~J.} \bblyear{1978}.
\newblock \bbltitle{The dielectric constant of a composite material --- {A}
  problem in classical physics}.
\newblock {\em \bbljournal{Physics Reports}\/} \bblvolume{43}\penalty0
  (\bblnumber{9}):\penalty0 \bblpages{377--407}.
\showEXTRA{%
\showCODEN{\bblCODEN{PRPLCM}}
\showISSN{\bblISSN{0370-1573}}
\showbibdate{\bblbibdate{Fri Nov 5 13:45:09 MST 1999}}
\showoriginalref{Bergman D.J. 1978 The dielectric constant of a composite
  material - A problem in classical physics. {\em Phys. Rep.} C {\bf 43},
  377--407.}
}

\bibitem[\protect\citeauthoryear{Bergman}{Bergman}{1980}]{Bergman:1980:ESM}
\bblauthor{Bergman, D.~J.} \bblyear{1980}.
\newblock \bbltitle{Exactly solvable microscopic geometries and rigorous bounds
  for the complex dielectric constant of a two-component composite material}.
\newblock {\em \bbljournal{Physical Review Letters}\/} \bblvolume{44}:\penalty0
  \bblpages{1285--1287}.
\showEXTRA{%
\showCODEN{\bblCODEN{PRLTAO}}
\showISSN{\bblISSN{0031-9007}}
\showbibdate{\bblbibdate{Fri Nov 5 13:45:09 MST 1999}}
\showoriginalref{Bergman D.J. 1980 Exactly solvable microscopic geometries and
  rigorous bounds for the complex dielectric constant of a two-component
  composite material. {\em Phys. Rev. Lett.} {\bf 44}, 1285--1287}
}

\bibitem[\protect\citeauthoryear{Bergman}{Bergman}{1982}]{Bergman:1982:RBC}
\bblauthor{Bergman, D.~J.} \bblyear{1982}.
\newblock \bbltitle{Rigorous bounds for the complex dielectric constant of a
  two-component composite}.
\newblock {\em \bbljournal{Annals of Physics}\/} \bblvolume{138}\penalty0
  (\bblnumber{1}):\penalty0 \bblpages{78--114}.
\showEXTRA{%
\showfjournal{\bblfjournal{Annals of Physics}}
\showCODEN{\bblCODEN{APNYA6}}
\showISSN{\bblISSN{0003-4916}}
\showMR
       {}%
       {\bblMRnumber{84e:78006}}%
       {\bblMRclass{78A25 (78A55 90C05)}}
\showbibdate{\bblbibdate{Fri Nov 5 13:45:09 MST 1999}}
\showoriginalref{Bergman D.J. 1982 Rigorous bounds for the complex dielectric
  constant of a two-component composite. {\em Ann. Phys.} {\bf 138}, 78--114.}
}

\bibitem[\protect\citeauthoryear{Bergman}{Bergman}{1986}]{Bergman:1986:EDC}
\bblauthor{Bergman, D.~J.} \bblyear{1986}.
\newblock \bbltitle{The effective dielectric coefficient of a composite medium:
  {Rigorous} bounds from analytic properties}.
\newblock In \bbleditor{J.~L. Ericksen}, \bbleditor{D.~Kinderlehrer},
  \bbleditor{R.~Kohn}, and \bbleditor{J.-L. Lions} (eds.), {\em
  \bblbooktitle{Homogenization and Effective Moduli of Materials and Media}},
  pp.\  \bblpages{27--51}. \bbladdress{Berlin~/ Heidelberg~/ London~/ etc.}:
  \bblpublisher{Springer-Verlag}.
\showEXTRA{%
\showISBN{\bblISBN{0-387-96306-5}}
\showLCCN{\bblLCCN{QA808.2 .H661 1986}}
\showbibdate{\bblbibdate{Fri Nov 5 13:45:09 MST 1999}}
\showoriginalref{Bergman D.J. 1986 The effective dielectric coefficient of a
  composite medium: Rigorous bounds from analytic properties. In {\em
  Homogenization and Effective Moduli of Materials and Media}, ed. by J.L.
  Ericksen, D. Kinderlehrer, R. Kohn, and J.-L. Lions., {\em The IMA volumes in
  mathematics and its applications} {\bf 1}, pp. 27--51, Springer-Verlag.}
}

\bibitem[\protect\citeauthoryear{Bernland, Luger, and Gustafsson}{Bernland
  et~al.}{2011}]{Bernland:2011:SRC}
\bblauthor{Bernland, A.}, \bblauthor{A.~Luger}, and \bblauthor{M.~Gustafsson}
  \bblyear{2011}.
\newblock \bbltitle{Sum rules and constraints on passive systems}.
\newblock {\em \bbljournal{Journal of Physics A: Mathematical and General}\/}
  \bblvolume{44}\penalty0 (\bblnumber{14}):\penalty0 \bblpages{145205}.
\showEXTRA{%
\showbibdate{\bblbibdate{Fri Sep 4 2015}}
}

\bibitem[\protect\citeauthoryear{Berreman}{Berreman}{1972}]{Berreman:1972:OSA}
\bblauthor{Berreman, D.~W.} \bblyear{1972}.
\newblock \bbltitle{Optics in stratified and anisotropic media: 4x4-matrix
  formulation}.
\newblock {\em \bbljournal{J. Opt. Soc. Am.}\/} \bblvolume{62}\penalty0
  (\bblnumber{4}):\penalty0 \bblpages{502--510}.
\showEXTRA{%
\showbibdate{\bblbibdate{Fri Sep 4 2015}}
}

\bibitem[\protect\citeauthoryear{Bruno}{Bruno}{1991}]{Bruno:1991:ECS}
\bblauthor{Bruno, O.~P.} \bblyear{1991}.
\newblock \bbltitle{The effective conductivity of strongly heterogeneous
  composites}.
\newblock {\em \bbljournal{Proceedings of the Royal Society of London. Series
  A, Mathematical and Physical Sciences}\/} \bblvolume{433}\penalty0
  (\bblnumber{1888}):\penalty0 \bblpages{353--381}.
\showEXTRA{%
\showfjournal{\bblfjournal{Proceedings of the Royal Society. London. Series A.
  Mathematical and Physical Sciences}}
\showCODEN{\bblCODEN{PRLAAZ}}
\showISSN{\bblISSN{0080-4630}}
\showMR
       {}%
       {\bblMRnumber{92c:82111}}%
       {\bblMRclass{82D30 (82B44)}}
\showbibdate{\bblbibdate{Fri Nov 5 13:45:09 MST 1999}}
\showoriginalref{Bruno O.P. 1991 The effective conductivity of strongly
  heterogeneous composites {\em Proc. Roy. Soc. Lond. A} {\bf 433}, 353--381.}
}

\bibitem[\protect\citeauthoryear{Carsey}{Carsey}{1992}]{Carsey:1992:MRS}
\bblauthor{Carsey, F.} \bblyear{1992}.
\newblock {\em \bbltitle{Microwave Remote Sensing of Sea Ice}}.
\newblock \bblseries{Geophysical Monograph Series}. \bblpublisher{Wiley}.
\showEXTRA{%
}

\bibitem[\protect\citeauthoryear{Cessenat}{Cessenat}{1996}]{Cessenat:1996:MME}
\bblauthor{Cessenat, M.} \bblyear{1996}.
\newblock {\em \bbltitle{Mathematical Methods in Electromagnetism}}.
\newblock \bbladdress{Singapore}: \bblpublisher{World Scientific}.
\newblock \bblpages{378} pp.
\showEXTRA{%
\showbibdate{\bblbibdate{Fri Sep 4 2015}}
}

\bibitem[\protect\citeauthoryear{Chaulet}{Chaulet}{2014}]{Chaulet:2014:ESP}
\bblauthor{Chaulet, N.} \bblyear{2014}.
\newblock \bbltitle{The electromagnetic scattering problem with generalized
  impedance boundary condition}.
\newblock {\em \bbljournal{ESAIM: Mathematical Modelling and Numerical
  Analysis}\/}.
\newblock \bblnote{To appear, see also arXiv:1312.1089 [math.AP]}.
\showEXTRA{%
}

\bibitem[\protect\citeauthoryear{Dell'Antonio, Figari, and
  Orlandi}{Dell'Antonio et~al.}{1986}]{DellAntonio:1986:ATO}
\bblauthor{Dell'Antonio, G.~F.}, \bblauthor{R.~Figari}, and
  \bblauthor{E.~Orlandi} \bblyear{1986}.
\newblock \bbltitle{An approach through orthogonal projections to the study of
  inhomogeneous or random media with linear response}.
\newblock {\em \bbljournal{Annales de l'Institut Henri Poincar{\'e}. Section B,
  Probabilit{\'e}s et statistiques}\/} \bblvolume{44}:\penalty0
  \bblpages{1--28}.
\showEXTRA{%
\showCODEN{\bblCODEN{AHPBAR}}
\showISSN{\bblISSN{0020-2347}}
\showbibdate{\bblbibdate{Fri Nov 5 13:45:09 MST 1999}}
\showoriginalref{Dell'Antonio G.F., Figari R., and Orlandi E. 1986 An approach
  through orthogonal projections to the study of inhomogeneous or random media
  with linear response. {\em Ann. Inst. Henri Poincar{\'e}} {\bf 44}, 1--28.}
}

\bibitem[\protect\citeauthoryear{Gesztesy and Tsekanovskii}{Gesztesy and
  Tsekanovskii}{2000}]{Gesztesy:2000:MVH}
\bblauthor{Gesztesy, F.} and \bblauthor{E.~Tsekanovskii} \bblyear{2000}.
\newblock \bbltitle{On matrix-valued {H}erglotz functions}.
\newblock {\em \bbljournal{Math. Nachr.}\/} \bblvolume{218}:\penalty0
  \bblpages{61--138}.
\showEXTRA{%
\showbibdate{\bblbibdate{Fri Sep 4 2015}}
}

\bibitem[\protect\citeauthoryear{Golden}{Golden}{1995}]{Golden:1995:BCP}
\bblauthor{Golden, K.} \bblyear{1995}.
\newblock \bbltitle{Bounds on the complex permittivity of sea ice}.
\newblock {\em \bbljournal{Journal of Geophysical Research (Oceans)}\/}
  \bblvolume{100}\penalty0 (\bblnumber{C7}):\penalty0 \bblpages{13699--13711}.
\showEXTRA{%
\showbibdate{\bblbibdate{Wed Sep 12 10:59:50 2001}}
\showoriginalref{Golden K.M. 1995 Bounds on the complex permittivity of sea
  ice, {\em J. Geophys. Res. (Oceans)} {\bf 100}, 13,699--13,711.}
}

\bibitem[\protect\citeauthoryear{Golden}{Golden}{2009}]{Golden:2009:CCM}
\bblauthor{Golden, K.} \bblyear{2009}.
\newblock \bbltitle{Climate change and the mathematics of transport in sea
  ice}.
\newblock {\em \bbljournal{Notices of the American Mathematical Society}\/}
  \bblvolume{56}\penalty0 (\bblnumber{5}):\penalty0 \bblpages{562--584}.
\showEXTRA{%
\showbibdate{\bblbibdate{Mon Oct 12 2015}}
}

\bibitem[\protect\citeauthoryear{Golden and Papanicolaou}{Golden and
  Papanicolaou}{1983}]{Golden:1983:BEP}
\bblauthor{Golden, K.} and \bblauthor{G.~Papanicolaou} \bblyear{1983}.
\newblock \bbltitle{Bounds for effective parameters of heterogeneous media by
  analytic continuation}.
\newblock {\em \bbljournal{Communications in Mathematical Physics}\/}
  \bblvolume{90}\penalty0 (\bblnumber{4}):\penalty0 \bblpages{473--491}.
\showEXTRA{%
\showfjournal{\bblfjournal{Communications in Mathematical Physics}}
\showCODEN{\bblCODEN{CMPHAY}}
\showISSN{\bblISSN{0010-3616}}
\showMR
       {}%
       {\bblMRnumber{84k:78006}}%
       {\bblMRclass{78A25 (28A99 58E99)}}
\showbibdate{\bblbibdate{Fri Nov 5 13:45:09 MST 1999}}
\showoriginalref{Golden K. and Papanicolaou, G.C. 1983 Bounds for effective
  parameters of heterogeneous media by analytic continuation. {\em Comm. Math.
  Phys.} {\bf 90}, 473--491.}
}

\bibitem[\protect\citeauthoryear{Golden and Papanicolaou}{Golden and
  Papanicolaou}{1985}]{Golden:1985:BEP}
\bblauthor{Golden, K.} and \bblauthor{G.~Papanicolaou} \bblyear{1985}.
\newblock \bbltitle{Bounds for effective parameters of multicomponent media by
  analytic continuation}.
\newblock {\em \bbljournal{Journal of Statistical Physics}\/}
  \bblvolume{40}\penalty0 (\bblnumber{5--6}):\penalty0 \bblpages{655--667}.
\showEXTRA{%
\showfjournal{\bblfjournal{Journal of Statistical Physics}}
\showCODEN{\bblCODEN{JSTPSB}}
\showISSN{\bblISSN{0022-4715}}
\showMR
       {\bblMRreviewer{Andrew Zardecki}}%
       {\bblMRnumber{87h:82063}}%
       {\bblMRclass{82A55}}
\showbibdate{\bblbibdate{Fri Nov 5 13:45:09 MST 1999}}
\showoriginalref{Golden K. and Papanicolaou G. 1985 Bounds for effective
  parameters of multicomponent media by analytic continuation. {\em J. Stat.
  Phys.} {\bf 40}, 655--667.}
}

\bibitem[\protect\citeauthoryear{Grabovsky}{Grabovsky}{1998}]{Grabovsky:1998:EREa}
\bblauthor{Grabovsky, Y.} \bblyear{1998}.
\newblock \bbltitle{Exact relations for effective tensors of polycrystals. {I}:
  {Necessary} conditions}.
\newblock {\em \bbljournal{Archive for Rational Mechanics and Analysis}\/}
  \bblvolume{143}\penalty0 (\bblnumber{4}):\penalty0 \bblpages{309--329}.
\showEXTRA{%
\showCODEN{\bblCODEN{AVRMAW}}
\showISSN{\bblISSN{0003-9527 (print), 1432-0673 (electronic)}}
\showbibdate{\bblbibdate{Fri Nov 5 13:45:09 MST 1999}}
\showoriginalref{Grabovsky Y., 1998 Exact relations for effective tensors of
  polycrystals. I: Necessary conditions. {\em Arch. Rational Mech. Anal.} {\bf
  143}, 309--329.}
}

\bibitem[\protect\citeauthoryear{Grabovsky}{Grabovsky}{2004}]{Grabovsky:2004:AGC}
\bblauthor{Grabovsky, Y.} \bblyear{2004}.
\newblock \bbltitle{Algebra, geometry and computations of exact relations for
  effective moduli of composites}.
\newblock In \bbleditor{G.~Capriz} and \bbleditor{P.~M. Mariano} (eds.), {\em
  \bblbooktitle{{Advances} in {Multifield} {Theories} of {Continua} with
  {Substructure}}}, \bblseries{Modelling and Simulation in Science, Engineering
  and Technology}, pp.\  \bblpages{167--197}. \bbladdress{Boston}:
  \bblpublisher{Birkh{\"a}user Verlag}.
\showEXTRA{%
\showISBN{\bblISBN{978-0-8176-4324-9}}
\showbibdate{\bblbibdate{Wed Apr 29 2015}}
}

\bibitem[\protect\citeauthoryear{Grabovsky, Milton, and Sage}{Grabovsky
  et~al.}{2000}]{Grabovsky:2000:ERE}
\bblauthor{Grabovsky, Y.}, \bblauthor{G.~W. Milton}, and \bblauthor{D.~S. Sage}
  \bblyear{2000}.
\newblock \bbltitle{Exact relations for effective tensors of composites:
  {Necessary} conditions and sufficient conditions}.
\newblock {\em \bbljournal{Communications on Pure and Applied Mathematics (New
  York)}\/} \bblvolume{53}\penalty0 (\bblnumber{3}):\penalty0
  \bblpages{300--353}.
\showEXTRA{%
\showCODEN{\bblCODEN{CPAMAT, CPMAMV}}
\showISSN{\bblISSN{0010-3640}}
\showbibdate{\bblbibdate{Fri Nov 5 13:45:09 MST 1999}}
\showoriginalref{Grabovsky Y., Milton G.W., and Sage D.S. 1999(?***) Exact
  relations for effective tensors of composites: Necessary conditions and
  sufficient conditions. {\em Comm. Pure Appl. Math.} submitted.}
}

\bibitem[\protect\citeauthoryear{Grabovsky and Sage}{Grabovsky and
  Sage}{1998}]{Grabovsky:1998:EREb}
\bblauthor{Grabovsky, Y.} and \bblauthor{D.~S. Sage} \bblyear{1998}.
\newblock \bbltitle{Exact relations for effective tensors of polycrystals.
  {II}: {Applications} to elasticity and piezoelectricity}.
\newblock {\em \bbljournal{Archive for Rational Mechanics and Analysis}\/}
  \bblvolume{143}\penalty0 (\bblnumber{4}):\penalty0 \bblpages{331--356}.
\showEXTRA{%
\showCODEN{\bblCODEN{AVRMAW}}
\showISSN{\bblISSN{0003-9527 (print), 1432-0673 (electronic)}}
\showbibdate{\bblbibdate{Fri Nov 5 13:45:09 MST 1999}}
\showoriginalref{Grabovsky Y., and Sage D.S., 1998 Exact relations for
  effective tensors of polycrystals. II: Applications to elasticity and
  piezoelectricity. {\em Arch. Rational Mech. Anal.} {\bf 143}, 331--356.}
}

\bibitem[\protect\citeauthoryear{{Guevara Vasquez}, Milton, and
  Onofrei}{{Guevara Vasquez} et~al.}{2011}]{Vasquez:2011:CCS}
\bblauthor{{Guevara Vasquez}, F.}, \bblauthor{G.~W. Milton}, and
  \bblauthor{D.~Onofrei} \bblyear{2011}.
\newblock \bbltitle{Complete characterization and synthesis of the response
  function of elastodynamic networks}.
\newblock {\em \bbljournal{Journal of Elasticity}\/}
  (\bblnumber{102}):\penalty0 \bblpages{31--54}.
\showEXTRA{%
\showbibdate{\bblbibdate{Wed Dec 22 2010}}
}

\bibitem[\protect\citeauthoryear{Gully, Lin, Cherkaev, and Golden}{Gully
  et~al.}{2015}]{Gully:2015:BCP}
\bblauthor{Gully, A.}, \bblauthor{J.~Lin}, \bblauthor{E.~Cherkaev}, and
  \bblauthor{K.~M. Golden} \bblyear{2015}.
\newblock \bbltitle{Bounds on the complex permittivity of polycrystalline
  materials by analytic continuation}.
\newblock {\em \bbljournal{Proceedings Royal Society A}\/}
  \bblvolume{471}\penalty0 (\bblnumber{2174}):\penalty0 \bblpages{20140702}.
\showEXTRA{%
\showbibdate{\bblbibdate{Sun Nov 1 2015}}
}

\bibitem[\protect\citeauthoryear{Gustafsson and {Sj\"{o}berg}}{Gustafsson and
  {Sj\"{o}berg}}{2010}]{Gustafsson:2010:SRP}
\bblauthor{Gustafsson, M.} and \bblauthor{D.~{Sj\"{o}berg}} \bblyear{2010}.
\newblock \bbltitle{Sum rules and physical bounds on passive metamaterials}.
\newblock {\em \bbljournal{New Journal of Physics}\/} \bblvolume{12}:\penalty0
  \bblpages{043046}.
\showEXTRA{%
\showbibdate{\bblbibdate{Tue Apr 9 2013}}
}

\bibitem[\protect\citeauthoryear{Hormander}{Hormander}{1990}]{Hormander:1990:ICA}
\bblauthor{Hormander, L.} \bblyear{1990}.
\newblock {\em \bbltitle{An introduction to complex analysis in several
  variables}\/} (\bbledition{Third} ed.).
\newblock \bbladdress{Amsterdam}: \bblpublisher{Elsevier}.
\showEXTRA{%
\showbibdate{\bblbibdate{Fri Sep 4 2015}}
}

\bibitem[\protect\citeauthoryear{Jackson}{Jackson}{1999}]{Jackson:1999:CE}
\bblauthor{Jackson, J.~D.} \bblyear{1999}.
\newblock {\em \bbltitle{Classical Electrodynamics}\/} (\bbledition{Third}
  ed.).
\newblock \bbladdress{New York}: \bblpublisher{John Wiley and Sons}.
\showEXTRA{%
\showISBN{\bblISBN{0-471-30932-X}}
\showLCCN{\bblLCCN{QC631 .J3 1998}}
\showbibdate{\bblbibdate{Wed Apr 24 2013}}
}

\bibitem[\protect\citeauthoryear{Kato}{Kato}{1995}]{Kato:1995:PTO}
\bblauthor{Kato, T.} \bblyear{1995}.
\newblock {\em \bbltitle{Perturbation theory for linear operators}}.
\newblock \bblseries{Classics in Mathematics}. \bbladdress{Berlin}:
  \bblpublisher{Springer-Verlag}.
\newblock \bblpages{xxi + 619} pp.
\showEXTRA{%
\showISBN{\bblISBN{3-540-58661-X}}
\showbibdate{\bblbibdate{Fri Sep 4 2015}}
}

\bibitem[\protect\citeauthoryear{Kirsch and Hettlich}{Kirsch and
  Hettlich}{2015}]{Kirsch:2015:MTT}
\bblauthor{Kirsch, A.} and \bblauthor{F.~Hettlich} \bblyear{2015}.
\newblock {\em \bbltitle{The Mathematical Theory of {Time-}Harmonic Maxwell's
  Equation}}.
\newblock \bblpublisher{Springer International Publishing}.
\newblock \bblpages{337} pp.
\showEXTRA{%
\showbibdate{\bblbibdate{Fri Sep 4 2015}}
}

\bibitem[\protect\citeauthoryear{Kor{\'a}nyi and Puk{\'a}nszky}{Kor{\'a}nyi and
  Puk{\'a}nszky}{1963}]{Koranyi:1963:HFP}
\bblauthor{Kor{\'a}nyi, A.} and \bblauthor{L.~Puk{\'a}nszky} \bblyear{1963}.
\newblock \bbltitle{Holmorphic functions with positive real part on
  polycylinders}.
\newblock {\em \bbljournal{American Mathematical Society Translations}\/}
  \bblvolume{108}:\penalty0 \bblpages{449--456}.
\showEXTRA{%
\showbibdate{\bblbibdate{Fri Nov 5 13:45:09 MST 1999}}
\showoriginalref{Kor{\'a}nyi A. and Puk{\'a}nszky L. 1963 Holmorphic functions
  with positive real part on polycylinders. {\em Amer. Math. Soc. Trans.} {\bf
  108}, 449--456.}
}

\bibitem[\protect\citeauthoryear{Landau, Lifshitz, and Pitaevski{\u\i}}{Landau
  et~al.}{1984}]{Landau:1984:ECM}
\bblauthor{Landau, L.}, \bblauthor{E.~Lifshitz}, and
  \bblauthor{L.~Pitaevski{\u\i}} \bblyear{1984}.
\newblock {\em \bbltitle{Electrodynamics of continuous media}}.
\newblock \bblpublisher{Pergamon}.
\showEXTRA{%
}

\bibitem[\protect\citeauthoryear{Lipton}{Lipton}{2000}]{Lipton:2000:OBE}
\bblauthor{Lipton, R.} \bblyear{2000}.
\newblock \bbltitle{Optimal bounds on electric field fluctuations for random
  composites}.
\newblock {\em \bbljournal{Journal of Applied Physics}\/}
  \bblvolume{88}\penalty0 (\bblnumber{7}):\penalty0 \bblpages{4287--4293}.
\showEXTRA{%
\showCODEN{\bblCODEN{JAPIAU}}
\showISSN{\bblISSN{0021-8979}}
\showbibdate{\bblbibdate{Fri May 26 13:45:09 MST 2000}}
}

\bibitem[\protect\citeauthoryear{Lipton}{Lipton}{2001}]{Lipton:2001:OIG}
\bblauthor{Lipton, R.} \bblyear{2001}.
\newblock \bbltitle{Optimal inequalities for gradients of solutions of elliptic
  equations occurring in two-phase heat conductors}.
\newblock {\em \bbljournal{SIAM Journal on Mathematical Analysis}\/}
  \bblvolume{32}\penalty0 (\bblnumber{5}):\penalty0 \bblpages{1081--1093}.
\showEXTRA{%
\showCODEN{\bblCODEN{SJMAAH}}
\showISSN{\bblISSN{1095-7154}}
\showbibdate{\bblbibdate{Fri May 26 13:45:09 MST 2000}}
}

\bibitem[\protect\citeauthoryear{Liu, Guenneau, and Gralak}{Liu
  et~al.}{2013}]{Liu:2013:CPP}
\bblauthor{Liu, Y.}, \bblauthor{S.~Guenneau}, and \bblauthor{B.~Gralak}
  \bblyear{2013}.
\newblock \bbltitle{Causality and passivity properties of effective parameters
  of electromagnetic multilayered structures}.
\newblock {\em \bbljournal{Physical Review B}\/} \bblvolume{88}\penalty0
  (\bblnumber{16}):\penalty0 \bblpages{165104}.
\showEXTRA{%
\showbibdate{\bblbibdate{Fri Sep 2 2015}}
}

\bibitem[\protect\citeauthoryear{Mattner}{Mattner}{2001}]{Mattner:2001:CDI}
\bblauthor{Mattner, L.} \bblyear{2001}.
\newblock \bbltitle{Complex differentiation under the integral}.
\newblock {\em \bbljournal{Nieuw Archief voor Wiskunde}\/}
  \bblvolume{5/2}\penalty0 (\bblnumber{2}):\penalty0 \bblpages{32--35}.
\showEXTRA{%
\showbibdate{\bblbibdate{Wed Nov 07 2015}}
}

\bibitem[\protect\citeauthoryear{Milton}{Milton}{1980}]{Milton:1980:BCD}
\bblauthor{Milton, G.~W.} \bblyear{1980}.
\newblock \bbltitle{Bounds on the complex dielectric constant of a composite
  material}.
\newblock {\em \bbljournal{Applied Physics Letters}\/} \bblvolume{37}\penalty0
  (\bblnumber{3}):\penalty0 \bblpages{300--302}.
\showEXTRA{%
\showCODEN{\bblCODEN{APPLAB}}
\showISSN{\bblISSN{0003-6951}}
\showbibdate{\bblbibdate{Fri Nov 5 13:45:09 MST 1999}}
\showoriginalref{Milton G.W. 1980 Bounds on the complex dielectric constant of
  a composite material. {\em Appl. Phys. Lett.} {\bf 37}, 300--302.}
}

\bibitem[\protect\citeauthoryear{Milton}{Milton}{1981a}]{Milton:1981:BCP}
\bblauthor{Milton, G.~W.} \bblyear{1981}a.
\newblock \bbltitle{Bounds on the complex permittivity of a two-component
  composite material}.
\newblock {\em \bbljournal{Journal of Applied Physics}\/}
  \bblvolume{52}\penalty0 (\bblnumber{8}):\penalty0 \bblpages{5286--5293}.
\showEXTRA{%
\showCODEN{\bblCODEN{JAPIAU}}
\showISSN{\bblISSN{0021-8979}}
\showbibdate{\bblbibdate{Fri Nov 5 13:45:09 MST 1999}}
\showoriginalref{Milton G.W. 1981 Bounds on the complex permittivity of a
  two-component composite material, {\em J. Appl. Phys.} {\bf 52}, 5286--5293.}
}

\bibitem[\protect\citeauthoryear{Milton}{Milton}{1981b}]{Milton:1981:BTO}
\bblauthor{Milton, G.~W.} \bblyear{1981}b.
\newblock \bbltitle{Bounds on the transport and optical properties of a
  two-component composite material}.
\newblock {\em \bbljournal{Journal of Applied Physics}\/}
  \bblvolume{52}\penalty0 (\bblnumber{8}):\penalty0 \bblpages{5294--5304}.
\showEXTRA{%
\showCODEN{\bblCODEN{JAPIAU}}
\showISSN{\bblISSN{0021-8979}}
\showbibdate{\bblbibdate{Fri Nov 5 13:45:09 MST 1999}}
\showoriginalref{Milton G.W. 1981 Bounds on the transport and optical
  properties of a two-component composite material. {\em J. Appl. Phys.} {\bf
  52}, 5294--5304.}
}

\bibitem[\protect\citeauthoryear{Milton}{Milton}{1987a}]{Milton:1987:MCEa}
\bblauthor{Milton, G.~W.} \bblyear{1987}a.
\newblock \bbltitle{Multicomponent composites, electrical networks and new
  types of continued fraction. {I}}.
\newblock {\em \bbljournal{Communications in Mathematical Physics}\/}
  \bblvolume{111}\penalty0 (\bblnumber{2}):\penalty0 \bblpages{281--327}.
\showEXTRA{%
\showfjournal{\bblfjournal{Communications in Mathematical Physics}}
\showCODEN{\bblCODEN{CMPHAY}}
\showISSN{\bblISSN{0010-3616}}
\showMR
       {\bblMRreviewer{V. Mastrangelo}}%
       {\bblMRnumber{89b:82084}}%
       {\bblMRclass{82A55 (73B99)}}
\showbibdate{\bblbibdate{Fri Nov 5 13:45:09 MST 1999}}
\showoriginalref{Milton G.W. 1987a Multicomponent composites, electrical
  networks and new types of continued fraction I. {\em Comm. Math. Phys.} {\bf
  111}, 281--327.}
}

\bibitem[\protect\citeauthoryear{Milton}{Milton}{1987b}]{Milton:1987:MCEb}
\bblauthor{Milton, G.~W.} \bblyear{1987}b.
\newblock \bbltitle{Multicomponent composites, electrical networks and new
  types of continued fraction. {II}}.
\newblock {\em \bbljournal{Communications in Mathematical Physics}\/}
  \bblvolume{111}\penalty0 (\bblnumber{3}):\penalty0 \bblpages{329--372}.
\showEXTRA{%
\showfjournal{\bblfjournal{Communications in Mathematical Physics}}
\showCODEN{\bblCODEN{CMPHAY}}
\showISSN{\bblISSN{0010-3616}}
\showMR
       {\bblMRreviewer{V. Mastrangelo}}%
       {\bblMRnumber{89b:82085}}%
       {\bblMRclass{82A55 (73B99 73F99 94C05)}}
\showbibdate{\bblbibdate{Fri Nov 5 13:45:09 MST 1999}}
\showoriginalref{Milton G.W. 1987b Multicomponent composites, electrical
  networks and new types of continued fraction II. {\em Comm. Math. Phys.} {\bf
  111}, 329--372.}
}

\bibitem[\protect\citeauthoryear{Milton}{Milton}{1990}]{Milton:1990:CSP}
\bblauthor{Milton, G.~W.} \bblyear{1990}.
\newblock \bbltitle{On characterizing the set of possible effective tensors of
  composites: {The} variational method and the translation method}.
\newblock {\em \bbljournal{Communications on Pure and Applied Mathematics (New
  York)}\/} \bblvolume{43}\penalty0 (\bblnumber{1}):\penalty0
  \bblpages{63--125}.
\showEXTRA{%
\showfjournal{\bblfjournal{Communications on Pure and Applied Mathematics}}
\showCODEN{\bblCODEN{CPAMAT, CPMAMV}}
\showISSN{\bblISSN{0010-3640}}
\showMR
       {\bblMRreviewer{John M. Ball}}%
       {\bblMRnumber{91c:73006}}%
       {\bblMRclass{73B27 (49S05 73K20)}}
\showbibdate{\bblbibdate{Fri Nov 5 13:45:09 MST 1999}}
\showoriginalref{Milton G.W. 1990 On characterizing the set of possible
  effective tensors of composites: The variational method and the translation
  method. {\em Comm. of Pure and Appl. Math} {\bf XLIII}, 63--125.}
}

\bibitem[\protect\citeauthoryear{Milton}{Milton}{2002}]{Milton:2002:TOC}
\bblauthor{Milton, G.~W.} \bblyear{2002}.
\newblock {\em \bbltitle{The Theory of Composites}}.
\newblock \bbladdress{Cambridge, United Kingdom}: \bblpublisher{Cambridge
  University Press}.
\newblock \bblpages{xxviii + 719} pp.
\showEXTRA{%
\showISBN{\bblISBN{0-521-78125-6}}
\showLCCN{\bblLCCN{TA418.9.C6M58 2001}}
\showbibdate{\bblbibdate{Fri Dec 5 07:58:58 2003}}
}

\bibitem[\protect\citeauthoryear{Milton}{Milton}{2016}]{Graeme:2016:ETC}
\bbleditor{Milton, G.~W.} (ed.) \bblyear{2016}.
\newblock {\em \bbltitle{Extending the {T}heory of {C}omposites to {O}ther
  {A}reas of {S}cience}}.
\newblock \bblpublisher{Milton-Patton Publishers}.
\newblock \bblpages{xx + 422} pp.
\showEXTRA{%
\showISBN{\bblISBN{978-1-4835-6919-2}}
}

\bibitem[\protect\citeauthoryear{Milton and Golden}{Milton and
  Golden}{1990}]{Milton:1990:RCF}
\bblauthor{Milton, G.~W.} and \bblauthor{K.~Golden} \bblyear{1990}.
\newblock \bbltitle{Representations for the conductivity functions of
  multicomponent composites}.
\newblock {\em \bbljournal{Communications on Pure and Applied Mathematics (New
  York)}\/} \bblvolume{43}\penalty0 (\bblnumber{5}):\penalty0
  \bblpages{647--671}.
\showEXTRA{%
\showfjournal{\bblfjournal{Communications on Pure and Applied Mathematics}}
\showCODEN{\bblCODEN{CPAMAT, CPMAMV}}
\showISSN{\bblISSN{0010-3640}}
\showMR
       {}%
       {\bblMRnumber{91c:78009}}%
       {\bblMRclass{78A30}}
\showbibdate{\bblbibdate{Fri Nov 5 13:45:09 MST 1999}}
\showoriginalref{Milton G.W. and Golden K. 1990 Representations for the
  conductivity functions of multicomponent composites. {\em Comm. Pure Appl.
  Math.} {\bf XLIII}, 647--671.}
}

\bibitem[\protect\citeauthoryear{Monk}{Monk}{2003}]{Monk:2003:FEM}
\bblauthor{Monk, P.} \bblyear{2003}.
\newblock {\em \bbltitle{Finite element Method for the Maxwell's Equation}}.
\newblock \bblpublisher{Oxford University Press}.
\newblock \bblpages{450} pp.
\showEXTRA{%
\showbibdate{\bblbibdate{Fri Sep 4 2015}}
}

\bibitem[\protect\citeauthoryear{Mujica}{Mujica}{1986}]{Mujica:1986:CAB}
\bblauthor{Mujica, J.} \bblyear{1986}.
\newblock {\em \bbltitle{Complex analysis in Banach spaces}}.
\newblock \bbladdress{Amsterdam}: \bblpublisher{Elsevier Science Publishers}.
\showEXTRA{%
\showbibdate{\bblbibdate{Fri Sep 4 2015}}
}

\bibitem[\protect\citeauthoryear{Naboko}{Naboko}{1996}]{Naboko:1996:ZTB}
\bblauthor{Naboko, S.~N.} \bblyear{1996}.
\newblock \bbltitle{Zygmund's theorem and the boundary behavior of operator
  {R}-functions}.
\newblock {\em \bbljournal{Functional Analysis and Its Applications}\/}
  \bblvolume{30}\penalty0 (\bblnumber{3}):\penalty0 \bblpages{211--213}.
\showEXTRA{%
\showbibdate{\bblbibdate{Fri Sep 4 2015}}
}

\bibitem[\protect\citeauthoryear{Nedelec}{Nedelec}{2001}]{Nedelec:2001:AEE}
\bblauthor{Nedelec, J.-C.} \bblyear{2001}.
\newblock {\em \bbltitle{Acoustic and Electromagnetic Equations: integral
  representations for harmonic problems}}.
\newblock \bbladdress{New York}: \bblpublisher{Springer Science \& Business
  Media}.
\newblock \bblpages{318} pp.
\showEXTRA{%
\showISBN{\bblISBN{978-0-387-95155-3}}
\showbibdate{\bblbibdate{Fri Sep 4 2015}}
}

\bibitem[\protect\citeauthoryear{Ola, P�iv�rinta, and Somersalo}{Ola
  et~al.}{2012}]{Ola:1993:IBV}
\bblauthor{Ola, P.}, \bblauthor{L.~P�iv�rinta}, and
  \bblauthor{E.~Somersalo} \bblyear{2012}.
\newblock \bbltitle{An inverse boundary value problem in electrodynamics}.
\newblock {\em \bbljournal{Duke mathematical Journal}\/}
  \bblvolume{70}\penalty0 (\bblnumber{3}):\penalty0 \bblpages{617--653}.
\showEXTRA{%
\showbibdate{\bblbibdate{Mon Nov 9 2015}}
}

\bibitem[\protect\citeauthoryear{Rudin}{Rudin}{1969}]{Rudin:1969:FTP}
\bblauthor{Rudin, W.} \bblyear{1969}.
\newblock {\em \bbltitle{Function Theory in Polydiscs}}.
\newblock \bbladdress{New York}: \bblpublisher{W. A. {Benjamin, Inc.}}
\newblock \bblpages{vii + 188} pp.
\showEXTRA{%
\showLCCN{\bblLCCN{QA331 .R86}}
\showbibdate{\bblbibdate{Fri Nov 5 13:45:09 MST 1999}}
\showoriginalref{Rudin W. 1969 {\em Function theory in polydisks}, Benjamin,
  New York.}
}

\bibitem[\protect\citeauthoryear{Shipman and Welters}{Shipman and
  Welters}{2013}]{Shipman:2013:RES}
\bblauthor{Shipman, S.~P.} and \bblauthor{A.~T. Welters} \bblyear{2013}.
\newblock \bbltitle{Resonant electromagnetic scattering in anisotropic layered
  media}.
\newblock {\em \bbljournal{Journal of Mathematical Physics}\/}
  \bblvolume{54}\penalty0 (\bblnumber{10}):\penalty0 \bblpages{103511}.
\showEXTRA{%
}

\bibitem[\protect\citeauthoryear{Stogryn}{Stogryn}{1987}]{Stogryn:1987:GTD}
\bblauthor{Stogryn, A.} \bblyear{1987}, \bblmonth{March}.
\newblock \bbltitle{An analysis of the tensor dielectnc constant of sea ice at
  microwave frequencies}.
\newblock {\em \bbljournal{Geoscience and Remote Sensing, IEEE Transactions
  on}\/} \bblvolume{GE-25}\penalty0 (\bblnumber{2}):\penalty0
  \bblpages{147--158}.
\showEXTRA{%
}

\bibitem[\protect\citeauthoryear{Uhlmann and Zhou}{Uhlmann and
  Zhou}{2014}]{Uhlmann:2014:REO}
\bblauthor{Uhlmann, G.} and \bblauthor{T.~Zhou} \bblyear{2014}.
\newblock \bbltitle{Inverse electromagnetic problems}.
\newblock In {\em \bblbooktitle{Encyclopedia of Applied and Computational
  Mathematics}}. \bblpublisher{Springer-Verlag}.
\showEXTRA{%
\showbibdate{\bblbibdate{Fri Sep 4 2015}}
}

\bibitem[\protect\citeauthoryear{Vladimirov}{Vladimirov}{2002}]{Vladimirov:2002:MTG}
\bblauthor{Vladimirov, V.} \bblyear{2002}.
\newblock {\em \bbltitle{Methods of the theory of generalized functions}}.
\newblock \bbladdress{London and New York}: \bblpublisher{Taylor and Francis}.
\showEXTRA{%
\showISBN{\bblISBN{0-415-27356-0}}
\showbibdate{\bblbibdate{Tue Oct 13 2015}}
}

\bibitem[\protect\citeauthoryear{Welters}{Welters}{2011a}]{Welters:2011:ERF}
\bblauthor{Welters, A.} \bblyear{2011}a.
\newblock \bbltitle{On explicit recursive formulas in the spectral perturbation
  analysis of a {J}ordan block}.
\newblock {\em \bbljournal{SIAM Journal on Matrix Analysis and Applications}\/}
  \bblvolume{32}:\penalty0 \bblpages{1--22}.
\showEXTRA{%
\showbibdate{\bblbibdate{Fri Sep 4 2015}}
}

\bibitem[\protect\citeauthoryear{Welters}{Welters}{2011b}]{Welters:2011:TSL}
\bblauthor{Welters, A.} \bblyear{2011}b.
\newblock \bbltitle{On the mathematics of slow light}.
\newblock \bbltype{Ph.D. thesis}, University of California, Irvine.
\showEXTRA{%
}

\bibitem[\protect\citeauthoryear{Welters, Avniel, and Johnson}{Welters
  et~al.}{2014}]{Welters:2014:SLL}
\bblauthor{Welters, A.}, \bblauthor{Y.~Avniel}, and \bblauthor{S.~G. Johnson}
  \bblyear{2014}.
\newblock \bbltitle{Speed-of-light limitations in passive linear media}.
\newblock {\em \bbljournal{Phys. Rev. A}\/} \bblvolume{90}:\penalty0
  \bblpages{023847}.
\showEXTRA{%
\showbibdate{\bblbibdate{Fri Sep 4 2015}}
}

\bibitem[\protect\citeauthoryear{Willemse and Caspers}{Willemse and
  Caspers}{1979}]{Willemse:1979:ECP}
\bblauthor{Willemse, M. W.~M.} and \bblauthor{W.~J. Caspers} \bblyear{1979}.
\newblock \bbltitle{Electrical conductivity of polycrystalline materials}.
\newblock {\em \bbljournal{Journal of Mathematical Physics}\/}
  \bblvolume{20}\penalty0 (\bblnumber{8}):\penalty0 \bblpages{1824--1831}.
\showEXTRA{%
\showCODEN{\bblCODEN{JMAPAQ}}
\showISSN{\bblISSN{0022-2488}}
\showbibdate{\bblbibdate{Mon Mar 20 13:45:09 MST 2000}}
}

\bibitem[\protect\citeauthoryear{Zettl}{Zettl}{2005}]{Zettl:2005:SLT}
\bblauthor{Zettl, A.} \bblyear{2005}.
\newblock {\em \bbltitle{{S}turm-{L}iouville {T}heory}}.
\newblock \bblpublisher{American Mathematical Society}.
\showEXTRA{%
}

\end{thebibliography}

\end{document}